\documentclass[]{article}

\usepackage[english]{babel}
\usepackage[nottoc]{tocbibind}
\usepackage[utf8]{inputenc}
\usepackage[english]{babel}
\usepackage{mathrsfs} 
\usepackage{amsfonts} 
\usepackage{mathabx}
\usepackage{amsthm}
\usepackage{amssymb}
\usepackage{tikz-cd} 
\usepackage{graphicx}

\newtheorem{theorem}{Theorem}[section]
\newtheorem{corollary}[theorem]{Corollary}
\newtheorem{lemma}[theorem]{Lemma}
\theoremstyle{definition}
\newtheorem{definition}[theorem]{Definition}
\newtheorem{proposition}[theorem]{Proposition}
\theoremstyle{remark}
\newtheorem*{remark}{Remark}

\title{Factor system for graphs and combinatorial HHS}
\author{Jihoon Park}

\begin{document}

\maketitle

\begin{abstract}
 We relaxe the constraint on the domains of combinatorial HHS machinery so combinatorial HHS machinery works for most cubical curve graphs. As an application we extend the factor system machinery of the CAT(0) cube complex to the quasi-median graphs.  
 
\end{abstract}
\section{Introduction}

 A hierarchically hyperbolic space(HHS) is a framework that simultaniously explains geometry of mapping class groups of finite type surfaces and CAT(0) cube complexes. After Behrstock--Hagen--Sisto introduced the definition and properties of HHSs \cite{BHS17,BHS19}, it was found that HHS structure is useful to study its coarse geometry and group structure. Although HHS structure covers wide range of spaces and groups (\cite{BHS17,BHS19,BR18,Hug22,BR20,HS20,RS20}), verifying whether a given space or group admits HHS structure directly from the definition is usually challenging.\par
 To remedy this, Behrstock--Hagen--Martin--Sisto \cite{BHMS20} provided a simpler criterion for hierarchical hyperbolicity, called the combinatorial hierarchically hyperbolic space (CHHS). This criterion has proven useful to study HHS structures of various groups, see \cite{BHMS20,DDLS21,HMS21,Rus21}. Also the CHHS structure can be formulated from the HHS structure under mild assumptions \cite{HMS23}, and understanding the CHHS machinery partly reveals mysterious nature of the HHS structure. The axioms of the CHHS, however, may not be optimal one as it does not cover some CAT(0) cube complexes with well-studied counterparts of the curve graphs; crossing and contact graphs, even when it admits an HHS structure. (One can find detailed discussion on the CAT(0) cube complex for which the CHHS machinery fails in [Section 10 of \cite{HMS23}].) \par 
 Nevertheless, when we consider a CAT(0) cube complexes $X$ with a factor system, the system of crossing graphs of each domains in the factor system is encoded in the crossing graph $\Delta X$ as intersections of link graphs. If we relax the constraint on the domains of the CHHS machinery by intersections of link graphs rather than the link graph of some clique graphs, then such system of subgraphs provides more adaptible combinatorial takes of HHS structures. We will call such system of subgraphs a $graph$ $factor$ $system$.  We can extend the axioms of the CHHS machinery to the simplicial graphs with a graph factor system. Now we state our main result: 
 \begin{theorem}[Main theorem]
 	Let $X$ be a simplicial graph with a graph factor system $\mathfrak{F}$ and let $W$ be an $X$-graph. If $(X,W,\mathfrak{F})$ satisfies the extended axioms of CHHS, then $(W,\mathfrak{F})$ admits an HHS structure.
 \end{theorem} 
 The CHHS with factor system covers most well-studied HHS with curve graph models including the curve graph of surface and the crossing graph of the CAT(0) cube complexes, and we believe that the CHHS with factor system could be useful to study new HHS structure of various spaces, especially for the spaces having a hierarchical system of subspaces. \par 
 As an application, we extend the factor system criterion of the CAT(0) cube complex developed by \cite{BHS17} to the quasi-median graphs. 
\section{Definitions and background}
\subsection{Graphs}
Let $X$ be a simplicial graph.
\begin{definition}[Clique, join, link]
	We say that $Y\subset V(X)$ is a $clique$ in $X$ if every pair of vertices in $Y$ are adjacent in $X$. If there is no confusion, we also mean clique by the complete subgraph of $X$ induced by the clique. A clique in $X$ is maximal if it is not a proper subgraph of other clique of $X$.
	Let $Y, Z\subset X$ be two disjoint induced subgraphs of $X$. If every vertex of $Y$ is adjacent to every vertex of $Z$, then the $join$  of $Y$ and $Z$, denoted by $Y\ast Z$, is the induced subgraph of $X$ spanned by the vertex set $Y^{0}\cup Z^{0}$. Given a subgraph $Y$ of $X$, the $link$ of $Y$, denoted by $lk(Y)$, is the induced subgraph of $X$ spanned by the vertices of $X$ that are adjacent to all the vertices of $Y$.
\end{definition}
\subsection{Hierarchically hyperbolic space}
    We recall from \cite{BR20} the definition of the hierarchically hyperbolic space which is devided into structural information part, called the $proto$-$hierarchy \ structure$, and geometric information of proto-hierarchy structure to induce HHS.
 	\begin{definition}
 		Let $X$ be a quasi-geodesic space and $E>0$. An $E$-$proto$-$hierarchy$ $structure$ on $X$ is an index set $\mathfrak{S}$ and a set $\{ CU : U\in \mathfrak{S} \}$ of geodesic spaces $(CU, d_{U})$ such that the following conditions are satisfied. \par 
 		
 		\begingroup
 		\renewcommand\labelenumi{(\theenumi)}
 		\begin{enumerate}
 			\item (Projections.) There is a set $\{ \pi_{U} : X \rightarrow 2^{CU} \}$ of $projections$ sending points in $X$ to sets of diameter bounded by  $E$ in the various $U\in \mathfrak{S}$. Moreover, for all $U\in\mathfrak{S}$, the coarse map $\pi_{U}$ is $(E,E)$-coarsely Lipschitz and $CU\subseteq N_{E}(\pi_{U}(X))$ .
 			\item (Nesting.) $\mathfrak{S}$ is equipped with a partial order $\sqsubseteq$, and either $\mathfrak{S}=\emptyset$ or $\mathfrak{S}$ contains a unique $\sqsubseteq$-maximal element. When $U\sqsubseteq V$, we say $U$ is $nested$ in $V$. For each $U\in\mathfrak{S}$, we denote by $\mathfrak{S}_{U}$ the set of all  $V\in\mathfrak{S}$ such that $V\sqsubseteq U$. Moreover, for all $U,V\in\mathfrak{S}$ with $U\sqsubsetneq V$ there is a specified subset $\rho^{U}_{V}\subset CV$ with diam$_{CV}(\rho^{U}_{V})\leq  E$. There is also a projection map $\rho^{V}_{U}$ from $CV\setminus N_{E}(\rho^{U}_{V})$ into $2^{CU}$.
 			\item (Orthogonality.) The set $\mathfrak{S}$ has a symmetric, anti-reflexive relation $\perp$ called $orthogonality$. Whenever $U\sqsubset V$ and $V\perp W$, we require that $U\perp W$. Also, if $U\perp V$, then $U$ and $V$ are not $\sqsubseteq$-comparable. 
 			\item (Transversality.) If $U,V\in\mathfrak{S}$ are not orthogonal and neither is nested in the other, then we say $U,V$ are $transverse$, denoted by $U \pitchfork V$. Moreover, if $U\pitchfork V$, then there are non-empty sets $\rho^{U}_{V}\subseteq CV$ and $\rho^{V}_{U}\subseteq CU$, each of diameter at most $E$
 		\end{enumerate}
 	
 	\endgroup
 	\noindent
 	\end{definition}
 	
 	We use $\mathfrak{S}$ to denote the proto-hierarchy structure, including the index set $\mathfrak{S}$, collection of associated spaces $\{ CU : U \in \mathfrak{S} \}$, projections $\{\pi_{U} : U\in\mathfrak{S} \}$ and relations $\sqsubseteq, \perp, \pitchfork$. We call the elements of $\mathfrak{S}$ the $domains$ of $\mathfrak{S}$ and call the set $\rho^{U}_{V}$ the $relative \ projection$ from $U$ to $V$. The number $E$ is called the $hierarchy \ constant$ for $\mathfrak{S}$.
 	
 	\begin{definition}
 		An $E$-proto-hierarchy structure $\mathfrak{S}$ on a quasi-geodesic space $\mathcal{X}$ is an $E$-$hierarchically \ hyperbolic \ space \ structure$($E$-HHS structure) on $\mathcal{X}$ if it satisfies the following additional conditions. 

\begingroup
\renewcommand\labelenumi{(\theenumi)}
\begin{enumerate}
	\item (Hyperbolicity.) For each $U\in\mathfrak{S}$, $CU$ is $E$-hyperbolic.
	\item (Finite complexity.) Any set of pairwise $\sqsubseteq$-comparable elements has cardinality at most $E$.
	
	\item (Containers.) If $T\in\mathfrak{S}$ and $U\in\mathfrak{S}_{T}$ for which $\{ V\in\mathfrak{S}_{T} \mid V\perp U \}\neq \emptyset$, then there exists $W\in\mathfrak{S}_{T}-\{T\}$, called the $container \ of \ U \ in \ T$, such that whenever $V\perp U$ and $V\sqsubseteq T$, we have $V\sqsubseteq W$. 
	
	\item (Bounded geodesic image.) For all $U,V\in\mathfrak{S}$ with $V\sqsubsetneq U$ and all $x,y\in\mathcal{X}$, if $d_{V}(\pi_{V}(x),\pi_{V}(y))\geq E$, then every $CU$-geodesic from $\pi_{U}(x)$ to $\pi_{U}(y)$ must intersect the $E$-neighbourhood of $\rho^{V}_{U}$.
	\item (Consistency.) If $U\pitchfork V$, then
	\begin{center}
		min$\{d_{V}(\pi_{V}(x),\rho^{U}_{V}), d_{U}(\pi_{U}(x),\rho^{V}_{U})  \}\leq E $
	\end{center}
   for all $x\in \mathcal{X}$. \par 
   Also if $U\sqsubsetneq V$, then 
   \begin{center}
   	min$\{d_{V}(\pi_{V}(x),\rho^{U}_{V}), diam_{U}(\pi_{U}(x)\cup \rho^{V}_{U}(\pi_{V}(x))  \}\leq E $
   	\end{center}
   for all $x\in \mathcal{X}$.
   Finally, if $U\sqsubseteq V$, then $d_{W}(\rho^{U}_{W},\rho^{V}_{W})\leq E$ whenever $W\in\mathfrak{S}$ satisfies either $V\sqsubsetneq W$ or $V\pitchfork W$ and $W$ is not orthogonal to $U$.
 \item (Large links.) Let $U\in \mathfrak{S}$ and $x,y\in \mathcal{X}$. Let $N=E d_{U}(\pi_{U}(x),\pi_{U}(y))+E$. Then there exists $\{ T_{i} \}_{i=1,...,\lfloor N\rfloor}\subseteq \mathfrak{S}_{U}-\{U\}$ such that for all $T\in\mathfrak{S}_{U}-\{U\}$, either $T\sqsubseteq T_{i}$ for some $i$, or $d_{T}(\pi_{T}(x),\pi_{T}(y))\leq E$. 
 \item (Partial Realization.) If $\{ U_{j}\}$ is a finite collection of pairwise orthogonal elements of $\mathfrak{S}$ and let $p_{j}\in  CU_{j}$ for each $j$, then there exists $x\in \mathcal{X}$ so that: 
 \begin{itemize}
 	\item $d_{U_{j}}(\pi_{U_{j}}(x), p_{j})\leq E$ for all $j$,
 	\item for each $j$ and each $V\in\mathfrak{S}$ such that $U_{j}\sqsubsetneq V$ or $U_{j}\pitchfork V$, we have $d_{V}(\pi_{V}(x),\rho^{U_{j}}_{V})\leq E$.
 \end{itemize}
\item (Uniqueness.) For each $\kappa\geq 0$, there exists $\theta_{u}=\theta_{u}(\kappa)$ such that if $x,y\in \mathcal{X}$ and $d_{\mathcal{X}}(x,y)\geq \theta_{u}$, then there exists $U\in\mathfrak{S}$ such that $d_{U}(\pi_{U}(x),\pi_{U}(y))\geq \kappa$.
\end{enumerate}
\endgroup
\noindent

\end{definition}

 We recall a useful result of HHS for later use, the realisation theorem which explains how the points in HHS and its coordinate tuples are related.

\begin{definition}[Consistent tuple]
	Let $\kappa\geq 0$ and let $\vec{b}\in \Pi_{U\in \mathfrak{S}} 2^{CU}$ be a tuple such that for each $U\in\mathfrak{S}$, the $U\textnormal{-}$coordinate $b_{U}$ has a diameter bounded above by $\kappa$. Then $\vec{b}$ is called a $\kappa \textnormal{-} consistent \ tuple$ if for all $V,W\in\mathfrak{S}$, we have \par 
	
	\begin{center}
		min$\{d_{V}(b_{V},\rho^{W}_{V}), d_{W}(b_{W},\rho^{V}_{W})  \}\leq \kappa $
	\end{center}
	whenever $V\pitchfork W$, and \par 
	
	\begin{center}
		min$\{d_{W}(b_{W},\rho^{V}_{W}), diam_{V}(b_{V}\cup \rho^{W}_{V}(b_{W}))  \}\leq \kappa $
	\end{center}
	whenever $V\sqsubsetneq W$.

\end{definition}

  For each point $x\in X$, the tuple $(\pi_{U}(x))_{U\in\mathfrak{S}}$ is always consistent. Conversely, all consistent tuples are coarsely the coordinate tuple of points in $X$.  
\begin{theorem}[The realisation of consistent tuples, \cite{BHS19}]
	Let $(X,\mathfrak{S})$ be an HHS. There exists a function $\theta : [0,\infty)\rightarrow [0,\infty)$ so that if $(b_{U})_{U\in\mathfrak{S}}$ is a $\kappa$-consistent tuple, then there exists a point $x\in X$ such that $d_{U}(x,b_{U})\leq \theta(\kappa)$ for all $U\in\mathfrak{S}$.
\end{theorem}
  Note that uniqueness condition for $(X,\mathfrak{S})$ implies that the $realisation \ point \ x$ for a consistent tuple $(b_{U})_{U\in\mathfrak{S}}$ provided by the realisation theorem, is coarsely unique.

 	\section{Combinatorial hierarchically hyperbolic factor system}
 	\subsection{Graphs with factor system and the axioms of CHHS}
 	In this subsection, we define a factor system on a simplicial graph $X$ and extend the notions and axioms of the CHHS theory to our setting.
 	\begin{definition}
 		Let $X$ be a simplicial graph. We define a $graph$ $factor$ $system$, denoted by $\mathfrak{F}$, as a collection of induced subgraphs of $X$ satisfying that \par
 	1. $X\in \mathfrak{F}$. \par 
 	2. For every vertex $v\in X$, $lk(v)\in\mathfrak{F}$. \par 
 	3. Whenever $F_{1}, F_{2}\in\mathfrak{F}$, $F_{1}\cap F_{2}\in \mathfrak{F}$ if it is non-empty. \par 
 	4. There exists $n>0$ such that if there exists a chain $F_{1}\subsetneq F_{2} \dots \subsetneq F_{k}$ of subgraphs in the system $F_{i}\in \mathfrak{F}$, $1\leq i\leq k$, it satisfies $k<n$. We call this $n$ the complexity of $\mathfrak{F}$. \par 
 	\end{definition}

 	\begin{definition}
 		An $X$-$graph$ is any graph $W$ whose vertex set is the set of all maximal cliques of $X$. We say that two maximal cliques of $X$ are $W$-$adjacent$ if the corresponding vertices of $W$ are adjacent in $W$.
 	\end{definition}
 	
 	The system of subgraphs $\mathfrak{F}$ is the candidate of an HHS structure for the metric graph $W$. For this reason we often refer to each subgraph in the system $\mathfrak{F}$ as a $domain$. For each domain $F\in\mathfrak{F}$, we also denote the induced system of subgraphs by $\mathfrak{F}_{F}=\{F'\in\mathfrak{F}\mid F'\subseteq F\}$.

 	\begin{definition}[Projection vertices, projection parts of domains, augmented graphs]
 		Let $X$ be a simplicial graph with a graph factor system $\mathfrak{F}$ and $W$ be an $X$-graph. For each non-maximal domain $F\in\mathfrak{F}-\{X\}$, we associate a new vertex $b_{F}$, called the $projection \ vertex \ for \ F$. Also denote the collection of all projection vertices of non-maximal domains by $B=\{ b_{F} \mid F\in \mathfrak{F}-\{ X \}  \}   $ and for each domain $F\in\mathfrak{F}$, denote the collection of all induced projection vertices by $B_{F}=\{ b_{F'} \mid F'\in \mathfrak{F}_{F}-\{F\} \}   $. \par 
 		If two domains $F,F'\in\mathfrak{F}$ satisfy $F'\subseteq lk(F)$, then we call $F$ and $F'$ are $orthogonal$ and denote by $F\perp F'$. In particular, if $lk(F)$ is non-empty, then $lk(F)$ is a domain and we denote this domain by $F^{\perp}$. \par 
 		We define a $projection \ part$ of $F$, denoted by $PF$, as $PF=F^{\perp}\cup (B_{F^{\perp}}\cup \{b_{F^{\perp}}\})\cup \{b_{F'}\in B\mid F\subseteq F'  \}$.\par
 		
 		The $augmented\ graph$ of $X$, denoted by $CX$, is the graph whose vertex set consists of $V(X)\cup B$ and two vertices are adjacent whenever \par 
 		1. two vertices $v,w\in V(X)$ that are adjacent in $X$.\par 
 		2. two vertices $v,w\in V(X)$ that are contained in $W$-adjacent maximal cliques. We often call this edge by $additional \ edge$.\par 
 		3. two vertices $v\in V(X), \ w\in B$ such that $w=b_{F}$ and $v\in F$.
 		\par
 		For each domain $F\in\mathfrak{F}$, the $augmented \ graph \ CF $ of the domain $F$ is an induced subgraph of $CX$ spanned by vertex set $(V(F)\cup B_{F})$.
 		
 	\end{definition} 
 	
 	\begin{remark}
 		The augmentation of each domain is slightly different compared to the original CHHS theory in \cite{BHMS20}. The reason is that we expect that the augmented graphs of properly nested domains are coned-off so the nested domains have uniformly bounded relative projections. Hence in \cite{BHMS20}, the nesting of essential domains $F\subsetneq F'$ requires a simplex $\Delta\subset F'$ satisfying $lk(\Delta)\cap F'=F$. In the graph factor system, such clique subgraphs may not exist so we need to take some additional vertices that coning-off every nested subgraphs.  
 	\end{remark}

 	\begin{definition}
 		
 		Let $F\in\mathfrak{F}$ be a domain. A $complement$ $graph$ of $F$, denoted by $Y_{F}$, is an induced subgraph of $CX$ spanned by the vertex set $(V(X)\cup B) -PF$.
 	\end{definition}
 	\par Now we extend the axioms of CHHS to the graph with the graph factor system. \par 
 	\begin{definition}
 		Let $(X,  W, \mathfrak{F})$ be a triple consisting of a simplicial graph $X$, an $X$-graph $W$ and a graph factor system $\mathfrak{F}$ on $X$. We define the $\delta$-$hyperbolicity$ $condition$ on the triple as follows:\par 
 		There exists $\delta>0$ such that for each domain $F\in \mathfrak{F}$, \par 
 		(1) the augmented graph $CF$ is $\delta$-hyperbolic graph and;\par 
 		(2) $CF$ admits $(\delta,\delta)$-quasi-isometric embedding into $Y_{F}$.
 	\end{definition}
 
  The hyperbolicity condition explains how the geometric and projection data of the expected HHS structure is encoded in the graph. Indeed if $\mathfrak{F}$ is a singleton set $\mathfrak{F}=\{ X \}$, then the condition tells that $W=CX$ is hyperbolic and hence admits trivial HHS structure without any additional assumption. But to induce actual HHS structure for general cases, we need more axioms to guarantee that the induced geometric and combinatorial structures on the nested domains are comparable. First we note that each domain admits an induced graph factor system. 
  \begin{lemma}
  	For each domain $F\in\mathfrak{F}$, the induced system $\mathfrak{F}_{F} = \{ F'\in \mathfrak{F} \mid F'\subseteq F \}$ is a graph factor system of the graph $F$.
  \end{lemma} 
  \begin{proof}
  First note that $F\in\mathfrak{F}_{F}$ by construction. For each $v\in F$, its link graph in $F$ satisfies $lk_{F}(v)=lk(v)\cap F \in \mathfrak{F}$ so $lk_{F}(v)\in\mathfrak{F}_{F}$. Also if $F_{1},F_{2}\in\mathfrak{F}_{F}\subset\mathfrak{F}$, then $F_{1}\cap F_{2}\in\mathfrak{F}$ and $F_{1}\cap F_{2}\subseteq F$, hence $F_{1}\cap F_{2}\in\mathfrak{F}_{F}$. Finally whenever  $F_{1}\subsetneq F_{2} \dots \subsetneq F_{k}$ is a chain of subgraphs in the system $\mathfrak{F}_{F}$, it is the chain of subgraphs in the system $\mathfrak{F}$ and hence $k$ is bounded by the complexity of $\mathfrak{F}$. 
  \end{proof}
  This observation shows that each domain $F\in\mathfrak{F}$ is itself a graph with the graph factor system $\mathfrak{F}_{F}$ of strictly smaller complexity. Let $W^{F}$ be some $F$-graph. We will denote the projection parts, complements graphs and augmented graphs of nested domains $H\in \mathfrak{F}_{F}$ by $PH_{F}$, $Y^{F}_{H}$ and $C_{F}H$ respectively. \par 
  Now we define an axiom which guarantees that for each domain, considered as the graph with the graph factor system, the induced CHHS structure is compatable with the structure of $(X,W,\mathfrak{F})$. 
  \par 
 
  \begin{definition}
  	Let $(X,  W, \mathfrak{F})$ be a triple of a simplicial graph $X$, an $X$-graph $W$ and a graph factor system $\mathfrak{F}$ on $X$. We say the triple $(X,  W, \mathfrak{F})$ satisfies the $hierarchy \ condition$ if 
  	for each domain $F\in\mathfrak{F}$, there exists a $F$-graph $W^{F}$ such that \par 
  	(1) for each domain $F'\in\mathfrak{F}_{F}$, we have $C_{F}F'=CF'$ and, \par 
  	(2) for each domain $F'\in \mathfrak{F}_{F}$ and for each (possibly empty) maximal clique $x$ in the projection parts $PF'_{F}$ , there is an induced map $x:W^{F'}\rightarrow W^{F}$ that assigns to each maximal clique $w\subset F'$ the maximal clique of $F$ that contains $x\ast w$ as a subclique.
  \end{definition}
  
  \begin{definition}
  	A $combinatorial \ HHS \ with \ factor \ system$(CHHF) is a triple $(X,W,\mathfrak{F})$ consists of a simplicial graph $X$, an $X$-graph $W$ and a graph factor system $\mathfrak{F}$ of $X$ that satisfies the $\delta$-hyperbolicity and hierarchy condition.
  \end{definition}
  
 Our goal is to show that if $(X,W,\mathfrak{F})$ is a CHHF, then $(W,\mathfrak{F})$ admits an $E$-HHS structure with $E$ depends only on the complexity $n$ of the graph factor system $\mathfrak{F}$ and the hyperbolicity constant $\delta$. To do this, first we need to show that $(W,\mathfrak{F})$ admits a proto-HHS structure. The projection maps will be defined in the augmented graph by nearest point projection maps, but to control the size of various images we need the hyperbolicity on the complement graphs and we will check this in the next section.

 \subsection{Proto-HHS structure and bounded geodesic image}
 Let $(X,W,\mathfrak{F})$ be a CHHF. The goal of this subsection is to show that $\mathfrak{F}$ induces a proth-HHS structure on $W$ and check the combinatorial bounded geodesic image for CHHF setting.
\par 
By Lemma 3.6, each non-maximal domain $F\in\mathfrak{F}$ admits a graph factor system $\mathfrak{F}_{F}$. We first check that the triple $(F,W^{F},\mathfrak{F}_{F})$ is a CHHF, where $W^{F}$ is given by the hierarchy condition.
 \begin{lemma}
 	For each domain $F\in\mathfrak{F}$, the triple $(F,W^{F},\mathfrak{F}_{F})$ is CHHF.\par 
 \end{lemma}
 \begin{proof}
 	By the hierarchy condition, for each domain $F'\in\mathfrak{F}_{F}$, the associated augmented graph $C_{F}F'$ is the same as $CF'$ which is $\delta$-hyperbolic by the hyperbolicity condition on $(X,W,\mathfrak{F})$. \par 
 	To show the second part of hyperbolicity condition for  $(F,W^{F},\mathfrak{F}_{F})$, consider the complement graph $Y^{F}_{F'}$ of $F'$ in $CF$. This graph is an induced subgraph of $CF$ spanned by vertices \begin{center}
 		
  $V(CF)-(\{v\in F \mid F'\subset lk(v) \}\cup \{b_{H}\in CF \mid F'\sqsubseteq H \ or \ H\perp F'  \})$ 	\end{center} which is presicely the vertex set of $Y_{F'}\cap CF$. Since $CF$ is an induced subgraph of $CX$ and $Y^{F}_{F'}$ is an induced subgraph of $CF$, the edges of $Y^{F}_{F'}$ and $Y_{F'}$ are induced ultimately from the graph $CX$, so $Y^{F}_{F'}$ is indeed an induced subgraph of $Y_{F'}$ and the inclusion $i : Y^{F}_{F'}\rightarrow Y_{F'}$ is 1-Lipschitz. Consider the following commutative diagram whose arrows are all inclusion.
 	
 	\[
 	\begin{tikzcd}
 	CF' \arrow[r, hook] \arrow[dr, hook] & Y^{F}_{F'} \arrow[d, hook] \\
 	& Y_{F'}
 	\end{tikzcd}
 	\]
 	We showed that the right arrow is 1-Lipschitz where the diagonal arrow is $(\delta,\delta)$-quasi-isometric embedding by the hyperbolicity condition. Hence the top inclusion is also uniformly quasi-isometric embedding. This shows the hyperbolicity condition for  $(F,W^{F},\mathfrak{F}_{F})$ .\par 
 	For each domain $F'\in\mathfrak{F}_{F}$, $F'\in\mathfrak{F}$ and the hierarchy condition of $(X,W,\mathfrak{F})$ induces some $F'$-graph $W^{F'}$. Associate such $W^{F'}$ to each domains $F'\in\mathfrak{F}_{F}$. For each pair of domains $H,K\in\mathfrak{F}_{F}$ with $H\sqsubsetneq K$ and each maximal clique subgraph $x\subset H^{\perp}\cap K$, taking  the induced map $x:W^{H}\rightarrow W^{K}$ from the hierarchy condition of $(X,W,\mathfrak{F})$.
 	The hierarchy condition of $(F,W^{F},\mathfrak{F}_{F})$ follows directly from the definition.    
 \end{proof}
 
 Lemma 3.9 tells that each domain of CHHF is itself a CHHF with strictly smaller complexity. One may expect various induction arguements on the complexity to understand $(X,W,\mathfrak{F})$, but to do this we need more connection between the geometric, combinatoric data of $X$ and its domains; the projection maps. For this purpose, we first check the hyperbolicity of the complement graphs of domains.

 \begin{lemma}
 	Let $F$ be a domain of the CHHF $(X,W,\mathfrak{F})$. Then $Y_{F}$ is uniformly hyperbolic graph with hyperbolicity constant depends only on $n$ and $\delta$.
 \end{lemma}
 
 Proof of this lemma is quite similar to that in \cite{BHMS20}. The difference is that, in the CHHS setting, each projection vertex is realized as a sub-clique of the given domains, while in the CHHF setting we blow-it up for each domains so we need slight modification on the leveled complement graphs.
 
 \begin{definition}
 	A $co$-$level$ $cl(F)$ of domain $F\in\mathfrak{F}$ is defined as follows;
 	\par 
 	Define $cl(X)=0$ and declare that a domain $F$ has co-level $n+1$ if it does not have co-level $\leq n$ but is properly contained in a domain with co-level $n$.  
 \end{definition}
 
 \begin{definition}
 	Let $F$ be a domain with co-level $n$. For each $0\leq k \leq n$, define $Y^{k}_{F}$ to be the induced subgraph of $CX$ spanned by the vertex set $V(CX)-(F^{\perp} \cup (B_{F^{\perp}}\cup \{b_{F^{\perp}}\}) \cup \{ b_{F'} \mid F\subset F' \ and \ cl(F') \leq k \} )$. 
 \end{definition}
 Note that $Y^{k}_{F}$ is obtained from $Y_{F}$ by adding projection vertices $\{b_{F'}\mid F\subseteq F', \ cl(F')>k \}$. Since $F\subseteq F'$ implies $cl(F')\leq cl(F)$, we have $Y^{cl(F)}_{F}=Y_{F}$. 
 We will prove the hyperbolicity of $Y_{F}$ by the following way: first we prove the hyperbolicity of $Y^{0}_{F}$ and inductively remove those additional projection vertices of co-level $k+1$ from $Y^{k}_{F}$ while hyperbolicity preserved, so that for each $0\leq k \leq cl(F)$, $Y^{k}_{F}$ is uniformly hyperbolic and in particular $Y_{F}$ is hyperbolic. 
 \begin{lemma}
 	$Y^{0}_{F}$ is uniformly quasi-isometric to $CX$. In particular, $Y^{0}_{F}$ is hyperbolic.
 \end{lemma}
\begin{proof}
    We show this by finding Lipschitz maps between $Y^{0}_{F}$ and $CX$. Since $Y^{0}_{F}$ is an induced subgraph of $CX$, the inclusion $i : Y^{0}_{F} \rightarrow CX$ is 1-Lipschitz. On the other hand, we can construct the quasi-inverse of the inclusion map $\phi : CX \rightarrow Y^{0}_{F} $ which is identity on $Y^{0}_{F}\subset CX$ and sends each element of $F^{\perp}\cup \{ b_{H} \mid H\perp F \}$ to a fixed element $b_{F}$. We can show that $\phi$ is a Lipschitz map by showing that for each edge $e\in CX$, the image of $e$ under $\phi$ has a uniformly bounded diameter in $Y^{0}_{F}$. If $e\subset Y^{0}_{F}$ or $e\subset CX\setminus Y^{0}_{F}=F^\perp \cup \{ b_{H} \mid H\perp F\}$, then $\phi(e)$ is either an edge or a vertex, hence we only need to consider the case the endpoints of $e$ are contained in the different graphs $Y^{0}_{F}$ and $CX\setminus Y^{0}_{F}$. If one end point is $b_{H}$ for some $H$ orthogonal to $F$, then by our construction of adjacency with projection vertices, the other end point must be a vertex of $H$. Since $H\subset F^{\perp}$, both end points of $e$ are contained in $CX\setminus Y^{0}_{F}$ and we are done. Hence the end point of $e$ in $CX\setminus Y^{0}_{F}$ lies in $F^{\perp}$. Denote the end points of $e$ by $v\in Y^{0}_{F}$ and $w\in F^{\perp}$. If $v\in X$, then we consider the few possible cases as follows.\par 
	Case 1. $v$ and $w$ are adjacent in $X$. In this case $lk(w)$ contains $v$ and $F\sqsubseteq lk(w)$, so we can join $\phi(v)=v$ and $\phi(w)=b_{F}$ by the edge path with end points $v-b_{lk(w)}-v'-b_{F}$ where $v'$ is some vertex in $F$. \par 
	Case 2. If $v$ and $w$ are adjacent by an additional edge, then there exists maximal cliques $x,y\subset X$ such that $v\in x, w\in y$ and $x,y$ are adjacent in $W$. In this case we can take a vertex $w'$ in $y\cap Y^{0}_{F}$ which is adjacent to $v$ by additional edge and $w'\in lk(w)$, so we have an edge path  $v-w'-b_{lk(w)}-v'-b_{F}$ where $v'\in F$.\par 
	The remaining case is when $v=b_{H}$. In this case $w\in H$ and hence $H\cap F^{\perp}\neq \emptyset$ while $b_{H}\in Y^{0}_{F}$ shows that $H\cap (X\setminus F^{\perp}) \neq \emptyset$, so $CH$ intersects both $F^{\perp}$ and $Y^{0}_{F}$ nontrivially. Since $CH$ is $\delta$-hyperbolic and hence connected, we can choose $v'\in CH\cap Y^{0}_{F}, w'\in CH \cap F^{\perp}$ that are adjacent. If $v'=b_{H'}$, then $v'=b_{H'}\in CH$ implies $H'\sqsubsetneq H$ and we can run the procedure again to choose some adjacent pair $v''\in CH'\cap Y^{0}_{F}\subset CH\cap Y^{0}_{F}$, $w'\in CH'\cap F^{\perp}\subset CH\cap F^{\perp}$. This process must be terminated by finite conplexity of $\mathfrak{F}$, so we can assume that $v'\in H\cap Y^{0}_{F}$ and $w'\in H\cap F^{\perp}$ which is the case we already done, so $v'$ and $b_{F}$ can be joined by edge path of length$\leq 4$. Since $v,v'\in Y^{0}_{F}$, $w,w'\in F^{\perp}$ and $v,v'$ are adjacent, we can join $\phi(v)=v$ and $\phi(w)=b_{F}$ by an edge path of length at most 5 so $\phi$ is lipschitz map as desired.
\end{proof}

 To remove the additional projection vertices while hyperbolicity is preserved, we need the following proposition [\cite{BHMS20}, Lemma 3.1].
 \begin{proposition}
  Let $Z$ be a $\delta$-hyperbolic graph and $\mathcal{V}\subset Z$ be a discrete collection of vertices. For each $v\in \mathcal{V}$, let $Z_{v}$ be an induced subgraph of $Z$ spanned by vertex set $V(Z)-{v}$. \par 
 	Suppose that for each $v\in \mathcal{V}$, \begin{itemize}
 		\item $lk(v)$ is $\delta$-hyperbolic graph.
 		\item $lk(v)$ admits $(\delta,\delta)-$quasi-isometric embedding into $Z_{v}$.
 	\end{itemize}
 	Then, there exists $\delta'$ depending only on $\delta$, so that the induced subgraph $Z_{V}$ of $Z$ spanned by vertex set $V(Z)-\mathcal{V}$ is $\delta'$-hyperbolic. \par 
 	Moreover, if $Q\subseteq Z$ is a $(\delta,\delta)$-quasi-isometrically embedded induced subgraph such that $Q$ contains the star of $v$ whenever $v\in \mathcal{V}\cap Q$, then the induced subgraph $Q_{\mathcal{V}}$ of $Z$ spanned by $V(Q)-(\mathcal{V}\cap Q)$ admits $(\delta',\delta')$-quasi-isometric embedding into $Z_{\mathcal{V}}$.
 \end{proposition}

To apply the proposition to $Y^{k}_{F}$, we need the hyperbolicity on the link graph of vertices that will be removed. The hyperbolicity of link graph of each $b_{F'}$ in $Y^{k}_{F}$ is expected from its similar shape with $CF'$, but in
the link graph of $b_{F'}$ no other projection vertices survive and hence is not quasi-isometric to $CF'$, so we need a slight modification of $Y^{k}_{F}$ in its quasi-isometry class. \par 
Let $Z^{k}_{F}$ be the graph obtained from $Y^{k}_{F}$ by adding some additional edges between the projection vertices, that is $Z^{k}_{F}$ is $Y^{k}_{F}$ together with the additional egdes joining every pair of projection vertices $b_{H},b_{H'}\in Y^{k}_{F}$ such that $H\sqsubsetneq H'$. Since each of such pair are joined by length 2 edge path in $Y^{k}_{F}$ via some vertex of $H$, $Z^{k}_{F}$ is uniformly quasi-isometric to $Y^{k}_{F}$. Also if $F\subseteq F'$ and $cl(F')=k+1$, then $F'\subsetneq H$ implies $cl(H)\leq k$ which implies that the link graph of $b_{F'}$ in $Z^{k}_{F}$ is $(F'\cap Y^{k}_{F}) \cup \{ b_{H} \mid H\sqsubsetneq F' \ and \  H\notperp F \}$. We show that such link graphs are quasi-isometric to $CF'$ and hence are uniformly hyperbolic in the following lemma. 
 \begin{lemma}
 	Let $0\leq k \leq cl(F)$. If $F\subseteq F'$ and $cl(F')= k+1$, then the link graph of $b_{F'}$ in $Z^{k}_{F}$ is hyperbolic.
 \end{lemma}
\begin{proof}
	First note that $Y^{k}_{F}$ and $Z^{k}_{F}$ have the same vertex set. Consider the link graph of $b_{F'}$ in $Z^{k}_{F}$, denoted by $lk_{Z^{k}_{F}}(b_{F'})$. Let $V$ be a vertex set of the graph $lk_{Z^{k}_{F}}(b_{F'})$ and consider the induced subgraph of $Y^{k}_{F}$ with the vertex set $V$, which we call $\Gamma$. Since the edges appeared in $lk_{Z^{k}_{F}}(b_{F'})$ but not in $\Gamma$ are precisely the pair of nesting comparable projection vertices contained in $\{ b_{H} \mid H\sqsubsetneq F' \ and \  H\notperp F \}$, the end points of each additional edges are uniformly close in $\Gamma$ and hence $lk_{Z^{k}_{F}}(b_{F'})$ and $\Gamma$ are uniformly quasi-isometric. (more precisely, if two projection vertices $b_{H},b_{K}$ are adjacent only in $lk_{Z^{k}_{F}}(b_{F'})$, then $H\subsetneq K$ or $K\subsetneq H$ and $H,K\notperp F$, hence $(H\cap K) \nsubset F^{\perp}$ which implies $(H\cap K) \cap Y^{k}_{F}\neq \emptyset$. Also we have  $(H\cap K)\subsetneq F'$ which implies $(H\cap K)\cap (F'\cap Y^{K}_{F})\neq\emptyset$. Since both $b_{H}$ and $b_{K}$ are adjacent to every vertices of $H\cap K$, the distance between $b_{H}$ and $b_{K}$ is at most 2 in $\Gamma$ via the edge path that passes through some point in $(H\cap K)\cap (F'\cap Y^{K}_{F})\neq\emptyset$.)
  
    We claim that $\Gamma$ is the leveled complement graph $Y^{0}_{F}$ considered in the induced CHHF triple $(F',W^{F'},\mathfrak{F}_{F'})$. As we already verified, the vertex set $V$ of $\Gamma$ consists of $(F'\cap Y^{k}_{F}) \cup \{ b_{H} \mid H\sqsubsetneq F' \ and \  H\notperp F \}$. We can rewrite $V$ by
     \begin{center}
    	$(F'-\{v\in F' \vert F\subseteq lk(v)\})\cup (B_{F'}-\{b_{H}\vert H\perp F\}) $
    \end{center}
    \begin{center}
    	 $  =CF'-((F^{\perp}\cap F')\cup \{b_{H}\in CF' \mid H\perp F \})$
    \end{center} 
    which is the desired vertex set of $Y^{0}_{F}$ considered in $(F',W^{F'},\mathfrak{F}_{F'})$. By Lemma 3.14, $\Gamma$ is quasi-isometric to $CF'$. Since $CF'$ is hyperbolic graph by the hyperbolicity condition of the CHHF triple $(F',W^{F'},\mathfrak{F}_{F'})$, $\Gamma$ is hyperbolic. $lk_{Z^{k}_{F}}(b_{F'})\subset Z^{k}_{F}$ is uniformly quasi-isometric to $\Gamma$ and hence is hyperbolic as desired.
\end{proof}

Since each $Z^{k}_{F}$ is quasi-isometric to $Y^{k}_{F}$, we know that $Z^{0}_{F}$ is hyperbolic. Also the colevel condition guarantees that the projection vertices with the same co-level on domains are pairwisely non-adjacent. \par  To use the proposition 3.15 inductively on $Z^{k}_{F}$, the only remaining part is showing that the link graphs $lk_{Z^{k}_{F}}(b_{F'})\subset Z^{k}_{F}$ admits uniform quasi-isometric embedding into $Z^{k}_{F}-\{b_{F'}\}$ for each co-level $k+1$ domain $F'$. \par

\begin{lemma}
	Let $F\in\mathfrak{F}-\{ X\}$ be a non-maximal domain and $0\leq k \leq cl(F)-1$. If $F'\in\mathfrak{F}$ is a domain with co-level $k+1$ and $F\sqsubseteq F'$, then the inclusion map from the link graph of the corresponding projection vertex $lk_{Z^{k}_{F}}(b_{F'})$ into $Z^{k}_{F}-\{b_{F'}\}$ is a quasi-isometric embedding.
\end{lemma}
\begin{proof}
	
	First note that if $F\sqsubsetneq F'$, then $Y_{F}\subset Y_{F'}$. If $v\in PF'$, then either $v\in X $ so $ F\subseteq F'\subseteq lk(v)$, or $v=b_{H}$ with either $F\sqsubsetneq F' \sqsubseteq H $ or $H\perp F'$. For the case $H\perp F'$, by orthogonality and $F\sqsubset F'$, we have $H\perp F$. So $v\in PF'$ implies $v\in PF$, that is, $PF'\subset PF$. Since $Y_{F}=CX-PF$ and $Y_{F'}=CX-PF'$, we have $Y_{F}\subset Y_{F'}$. Note also that for each additional projection vertex $b_{H}$ of $Y^{k}_{F}$ with $b_{H}\neq b_{F'}$, we have $b_{H}\in Y_{F'}$ since the co-level condition guarantees $F' \nsqsubseteq H$ and $F'\perp H$ implies $F\perp H$, contradicts to the assumption that $b_{H}\in Y^{k}_{F}$. Hence we have $Y^{k}_{F}-\{b_{F'}\}\subset Y_{F'}$.\par 
	
	Now consider the following commutative diagram
	\[\begin{tikzcd}
		lk_{Z^{k}_{F}}(b_{F'}) \arrow{r} \arrow[swap]{d} & CF' \arrow{d} \\
		Z^{k}_{F}-\{ b_{F'} \} \arrow{r} & Y_{F'}
	\end{tikzcd}
	\]
	whose arrows are all inclusion on the vertex set. \par 
	We showed that the top arrow is a quasi-isometry map in the proof of lemma 3.16, and the right arrow is quasi-isometric embedding by the hyperbolicity condition of $(X,W,\mathfrak{F})$. Since $Y^{k}_{F}-\{b_{F'}\}\subset Y_{F'}$ is an induced subgraph and hence the inclusion map is Lipschitz while $Z^{k}_{F}-\{ b_{F'}\}$ is quasi-isometric to $Y^{k}_{F}-\{b_{F'}\}$, we have that the bottom arrow is also a Lipschitz map. Finally the left arrow is an induced map and hence is Lipschitz. This conclude that the left inclusion map is a quasi-isometric embedding.

\end{proof}

 \begin{proof}[Proof of lemma 3.10]
 	Now we prove the hyperbolicity of $Y^{k}_{F}$ for $0\leq k \leq cl(F)$ by induction on $k$. \par
 	The initial case $Y^{0}_{F}$ is done in lemma 3.14. \par 
 	If $Y^{k}_{F}$ is $\delta$-hyperbolic, then $Z^{k}_{F}$ is also hyperbolic with uniform constant. Also each link graph of $b_{F'}$ in $Z^{k}_{F}$ with $cl(F')=k+1$ is uniformly hyperbolic and quasi-isometrically embedded in $Z^{k}_{F}-\{ b_{F'} \}$. Take the uniform constant that works for both $Z^{k}_{F}$ and $b_{F'}$ by $\delta'=\delta'(\delta)$. Then the proposition 3.15 implies that $Z^{k}_{F}-\{  b_{F'} \in Z^{k}_{F} \mid F\sqsubseteq F', \ cl(F')=k+1 \}=Z^{k+1}_{F}$ is $\delta''(\delta)$-hyperbolic, hence $Y^{k+1}_{F}$ is hyperbolic with hyperbolicity constant depends only on $\delta$. \par 
 	The induction shows that $Y^{cl(F)}_{F}=Y_{F}$ is hyperbolic with hyperbolicity constant depends only on $n$ and $\delta$.
 \end{proof}
 
 The hyperbolicity of the complement graphs explains how to define various coarsely well-defined projection maps. Now we state the proto-HHS structure on the CHHF $(X,W,\mathfrak{F})$.
 
 \begin{lemma}
 	Let $(X,W,\mathfrak{F})$ be a CHHF triple. Then $(W,\mathfrak{F})$ admits a proto-HHS structure.
 \end{lemma}
 
 \begin{proof} 
 	First note that we already verified the orthogonality axiom of $(W,\mathfrak{F})$.
 	\begin{enumerate}
 		\item (Projections.) First we define the projections for each domain.
 		If $F=X$, then the map $\pi_{X} : W \rightarrow 2^{CX}$ is simply an inclusion on graphs, sending a point $x\in W$, considered as a maximal clique of $X$, to $x\subseteq CX$ itself. We can also check that the projection map $\pi_{X}$ is coarsely well-defined Lipschitz map. Since $x$ is a clique subgraph of $X$ and $X$ is an induced subgraph of $CX$, $x$ is also the clique subgraph of $CX$ and hence have diameter at most 1. By construction of $CX$, if $x,y\in W$ are $W$-adjacent maximal cliques, then each vertex of $x$ is adjacent to each vertex of $y$ in $CX$. $\pi_{X}$ sends an edge $[x,y]$ of $W$ to a clique subgraph $x \ast y$ of $CX$, which shows that $\pi_{X}$ is a $(1,1)$-coarsely Lipschitz map. Since each additional vertices in $B$ is adjacent to at least one vertex of $X$ and $X\subseteq \pi_{X}(W)$, we have $CX\subseteq N_{2}(\pi_{X}(W))$. This shows that the map $\pi_{X}$ is the desired projection map. \par 
 		
 		For each domains $F\in\mathfrak{F}-\{X\}$, we define a projection map $\pi_{F}:W\rightarrow 2^{CF}$ as follows. Let $p_{F}:Y_{F}\rightarrow 2^{CF}$ be the coarse closest point projection, namely, $p_{F}(x)=\{ y\in CF \mid d_{Y_{F}}(x,y)\leq d_{Y_{F}}(x,CF)+1        \}$. Let $w$ be a vertex of $W$, considered as a maximal clique in $X$. Since $w$ is maximal, $w\cap Y_{F}$ is non-empty (otherwise $w\subset F^{\perp}$, which implies that for some $z\in F\neq \emptyset$, $z\in F\subset lk(w)$ induces a clique subgraph $w\ast z$ of $X$ that properly contains $w$, which violates the maximality of $w$.) and has diameter at most 1. We define $\pi_{F}(w)=p_{F}(w\cap Y_{F})$. We defined $\pi_{F}$ on the vertices of $W$, so define on the edge set of $W$ by taking $\pi_{F}(e)=\pi_{F}(x)\cup \pi_{F}(y)$ if $e$ is an edge of $W$ and $x,y$ are its endpoints. 
 		\par 
 		We need to show that the projection map $\pi_{F}$ is coarsely well-defined, uniformly Lipschitz map. First note that $Y_{F}$ is hyperbolic and $CF$ is quasi-isometrically embedded in $Y_{F}$, which implies that $CF$ is a quasi-convex subgraph of $Y_{F}$. Since coarse nearest point projections to quasi-convex subsets of hyperbolic spaces are coarsely Liptschitz maps, $p_{F}$ is uniformly Lipschitz. Since $(w\cap Y_{F})$ has diameter 1, its Lipschitz projection image $p_{F}(w\cap Y_{F})$ has uniformly bounded diameter and $\pi_{F}$ is uniformly coarsely well-defined. Also $x,y$ are $W$-adjacent maximal cliques, then $x\cup y$ forms a clique subgraph in $CX$, hence the diameter of $(x\cup y)\cap Y_{F}$ and its image $\pi_{F}([x,y])=p_{F}(x\cap Y_{F})\cup p_{F}(y\cap Y_{F})$ has uniformly bounded diameter. It concludes that $\pi_{F}$ is uniformly Lipschitz map as desired. \par 
 		Finally $CF\subseteq N_{2}(\pi_{F}(W))$ since for each vertex $v\in F$, any maximal clique $w\subset X$ that contains $v$ satisfies $v\in\pi_{F}(w)$ and $CF\subseteq N_{2}(F)$.
 		\par 
 		For each $F\in\mathfrak{F}$ and $w\in W$, we often call $\pi_{F}(w)$ the $projections$ or, to distinguish from the maps, the $coordinates$ of $w$ on $F$.
 		
 		\item (Nesting.) We define the nesting relation $\sqsubseteq$ on $\mathfrak{F}$ by $F\sqsubseteq F'$ whenever $F\subseteq F'$ as the subgraphs in $X$. It is clearly a partial order and have a unique $\sqsubseteq$-maximal element $X$.
 		\par For any domains $F,F'\in \mathfrak{F}$ with $F\sqsubsetneq F'$, we also define the relative projections $\rho^{F}_{F'}$ by the projection image of the projection vertices $p_{F'}(PF\cap Y_{F'})$. 
 		Since $F\sqsubsetneq F'$ implies that $b_{F}\in Y_{F'}$ so $PF\cap Y_{F'}\neq \emptyset$ and $\rho^{F}_{F'}$ is well-defined. Also the diameter of $PF\cap Y_{F'}$ is bounded by at most 4, since $F\sqsubsetneq F'\subset Y_{F'}$ and any vertices of $PF$ can be joined by an edge path of length 2 to some vertex in $F$. Since $p_{F'}$ is a uniformly lipschitz projection, $\rho^{F}_{F'}$ is uniformly bounded set.
 		\par Finally we define the downward relative projections between the nesting pairs $F\sqsubsetneq F'$ by $\rho^{F'}_{F}=p_{F}$ over $CF'\cap Y_{F}$ and takes the value $\emptyset$ otherwise.

 		\item(Orthogonality.) Since $F^{\perp}\cap F=\emptyset$, the orthogonality relation is anti-reflective. To check the symmetry, first note that for any pair of subgraphs $F,F'\subset X$, $F\subseteq F'$ implies $(F')^{\perp}\subseteq F^{\perp}$ since $v\in (F')^{\perp}$ implies $F\subseteq F'\subseteq lk(v)$ and hence $v\in F^{\perp}$. If $F'\perp F$, then $F'\subseteq F^{\perp}$ and hence $F\subseteq (F^{\perp})^{\perp}\subseteq (F')^{\perp}$ which implies $F\perp F'$.
 		
 		\par 
 		If $F\sqsubsetneq F'$ and $F''\perp F'$, then $F\perp F''$ since $F\subseteq F'\subseteq  (F'')^{\perp}$. \par 
 		Finally since $F$ and $F^{\perp}$ are disjoint, if $F\perp F'$ or equivalently $F'\subset F^{\perp}$, then $F$ and $F'$ are disjoint and hence not nesting comparable.
 		
 		\item(Transversality.) We define two domains $F, F'$ are transversal, denoted by $F \pitchfork F'$, if they neither nested to one another nor orthogonal. 
 		In this case, by construction of $Y_{F}$, the projection vertex $b_{F'}$ is contained in $Y_{F}$ so $PF\cap Y_{F'}\neq \emptyset$. Also $F\nsubseteq (F')^{\perp}$ implies that $F\cap Y_{F'}\neq \emptyset$. Hence $PF\cap Y_{F'}$ has diameter at most $4$ in $Y_{F'}$ via some vertex of $F$. We define the relative projections  $\rho^{F}_{F'}\subset CF'$ by $p_{F'}(PF\cap Y_{F'})$. Again $\rho^{F'}_{F}$ is uniformly bounded since $p_{F'}$ is a uniformly lipschitz map.
 		\par  
 		
 	\end{enumerate}

 \end{proof}
 
 To use the induction on the (proto)-HHS structure of non-maximal domains, we need to check that the induced proto-HHS structures are compatible with that of its domains. More precisely, we need to show that the various projections are uniformly coarsely coincide and the relations of the domains are preserved in the proto-HHS structure of its domains. Since the relations are purely graph property and can be easily checked, we only check the compatibility of projections. 
 
 \begin{lemma}
 	Let $F'\in\mathfrak{F}$ be a domain. For each $F\in\mathfrak{F}_{F'}-\{F'\}$, the induced complement graph $Y^{F'}_{F}=Y_{F}\cap CF'$ admits a uniform quasi-isometric embedding into $Y_{F}$. 
 \end{lemma}
 \begin{proof}
 	The idea of proof is similar to the proof of the hyperbolicity of $Y_{F}$. We will construct similar graphs with $Y^{k}_{F}$ and $Y^{k}_{F}\cap CF'$ which admits quasi-isometric embeddings, and using the 'moreover' part of proposition 3.14, inductively remove additional projection vertices while quasi-isometric embedding property is preserved. \par 
    Let $Q^{k}_{F}$ be an induced subgraph of $Y^{0}_{F}$ spanned by the vertices of $Y^{k}_{F}\cap (CF'\cup b_{F'})$. Then the initial case $Q^{0}_{F}$ has diameter at most 2. To see this, note that $F\sqsubsetneq F'$ impliess $cl(F')>0$ and in particular, $b_{F'}\in Y^{0}_{F}$. Hence $b_{F'}\in Q^{0}_{F}$ is adjacent to all the other vertices of $Q^{0}_{F}$. Indeed, whenever $k\leq cl(F')-1$, $Q^{k}_{F}$ has diameter at most 2. On the other hand, $Q^{cl(F)}_{F}$ is spanned by $Y^{cl(F)}_{F} \cap (CF'\cup b_{F'})=Y_{F}\cap CF'$. \par 
    We will argue by induction on $k$ such that for all $0\leq k \leq cl(F)$, the inclusion $Q^{k}_{F}\rightarrow Y^{k}_{F}$ is a uniform quasi-isometric embbedding. The initial case is automatic, since $Q^{0}_{F}$ is bounded set. \par 
    Before we run the induction process, we again slightly modify the graphs $Q^{k}_{F}$ and $Y^{k}_{F}$ in a way that assigns to each projection vertices appropriate edges so removing co-level $k+1$ projection vertices works with proposition 3.14. \par 
    Let $R^{k}_{F}$ be a graph obtained from $Q^{k}_{F}$ by adding edges between the projection vertices of the nesting comparable domains. Note that by the same argument used for $Z^{k}_{F}$, $R^{k}_{F}$ is quasi-isometric to $Q^{k}_{F}$ and the inclusion map from $Q^{k}_{F}\rightarrow Y^{k}_{F}$ extends to an induced map $R^{k}_{F}\rightarrow Z^{k}_{F}$. Moreover, the quasi-isometric embedding $Q^{k-1}_{F}\rightarrow Y^{k-1}_{F}$ induces the quasi-isometric embedding $R^{k-1}_{F}\rightarrow Z^{k-1}_{F}$.\par 
    By induction argument, we have that the induced map $R^{k-1}_{F}\rightarrow Z^{k-1}_{F}$ is a quasi-isometric embedding. Also we already showed that $Z^{k-1}_{F}$ is hyperbolic, the link graphs of the projection vertices of the domains containing $F$ and have co-level $k$ are uniformly hyperbolic, quasi-isometrically embedded. It only remains to show that for each projection vertex of $R^{k-1}_{F}$ whose domains has co-level $k$ and contains $F$, its link graph in $Z^{k-1}_{F}$ is contained in $R^{k-1}_{F}$. To show this, note that if $b_{H}\in R^{k-1}_{F}$, $F\sqsubseteq H$ and $cl(H)=k$, then since $b_{H}\in R^{k-1}_{F}\subseteq Y^{0}_{F}\cap CF'$, we have that $H\sqsubseteq F'$. Also the link graph of $b_{H}$ in $Z^{k-1}_{F}$ is spaned by the vertices contained in $H\cap Y_{F}$ and the projection vertices $\{b_{J}\in Y^{k-1}_{F} \vert J\sqsubsetneq H\}$. Since $H\subseteq F'$, $H\cap Y_{F}=H\cap CF' \cap Y_{F}$. Since $H\sqsubseteq F'$, whenever $J\sqsubsetneq H$, $J\sqsubsetneq F'$ and hence $\{b_{J} \in Y^{k-1}_{F}\vert J\sqsubsetneq H\}\subset Y^{k-1}_{F}\cap CF'$. Since $R^{k-1}_{F}$ is an induced subgraph of $Z^{k-1}_{F}$, the link graph of each $b_{H}$ is contained in $R^{k-1}_{F}$ as desired. By second part of proposition 3.14, the induced map $R^{k}_{F}\rightarrow Z^{k}_{F}$ is uniform quasi-isometric embedding and in particular, $Q^{cl(F)}_{F}(=Y_{F}\cap CF')\rightarrow Y^{cl(F)}_{F}(=Y_{F})$ is a quasi-isometric embedding. 
    
 \end{proof}
  The compatibility of the projection over the complement graphs then follows.
  \begin{lemma}
  	If $F\sqsubsetneq F'$, then $p^{F'}_{F} : Y^{F'}_{F}\rightarrow CF$, the projection of the induced complement graph, is uniformly coarsely coincide with the projection map $p_{F} : Y_{F}\rightarrow CF$. 
  \end{lemma}
  \begin{proof}
  	The proof is identical to [proposition 4.11, \cite{BHMS20}]. Let $v\in Y^{F'}_{F}$. Let $\gamma$ be a $Y^{F'}_{F}$-geodesic from $v$ to $p^{F'}_{F}(v)$. Since $Y^{F'}_{F}$ is an induced subgraph of $Y_{F}$, $\gamma$ is a uniform quasi-geodesic of $Y_{F}$. Hence $\gamma$ fellow-travels with a geodesic of $Y_{F}$ joining $v$ and $p^{F'}_{F}(v)$. Since $p^{F'}_{F}(v)\in CF$, $\gamma$ contains some point $u$ which lies at uniformly bounded distance, say C, from $p_{F}(v)$ (in $Y^{F'}_{F}$). The concatenation of the subpath of $\gamma$ from $v$ to $u$ and the geodesic from $u$ to $p_{F}(v)$ forms a path between $v$ and $p_{F}(v)$ with length at most $d_{Y^{F'}_{F}}(v,u)+C\geq \vert \gamma \vert$. Hence $d_{Y^{F'}_{F}}(u,p^{F'}_{F}(v))\leq C$ and in particular, $d_{Y^{F'}_{F}}(p_{F}(v),p^{F'}_{F}(v))\leq 2C$. Since both $p_{F}(v),p^{F'}_{F}(v)$ are contained in $CF$, we conclude that $p_{F}(v)$ and $p^{F'}_{F}(v)$ are uniformly close, as desired.
  \end{proof}
  
  It is direct from the definition and lemma 3.19 that the relative projections are uniformly coarsely coincide in all possibilities. so we only need to check the compatibility of the projection maps from $X$-graphs. 
 
 \begin{lemma}
 	If $x$ is a maximal clique subgraph contained in $F^{\perp}$, then the map $x : W^{F} \rightarrow W$ given by the hierarchy condition commutes with the projection map $\pi_{F'}$ for each domain $F'\in\mathfrak{F}_{F}$. That is, the following diagram uniformly coarsely commutes :
 	\[\begin{tikzcd}
 		W^{F} \arrow{r}{x} \arrow[swap]{d}{\pi^{F}_{F'}} & W \arrow{d}{\pi_{F'}} \\
 		C_{F}F' \arrow[r, equal] & CF'
 	\end{tikzcd}
 	\]
 \end{lemma}
 \begin{proof}

 	Let $\pi^{F}_{F'} $ be the projection map from $W^{F}$ to $C_{F}F'=CF'$. For any maximal clique $w\subset F$ and $F'\in\mathfrak{F}_{F}$, $w\cap Y^{F}_{F'}$ is non-empty and $\pi^{F}_{F'}(w)=p^{F}_{F'}(w\cap Y^{F}_{F'})$. On the other hand, $x(w)$ is a maximal clique of $X$ that contains $x\ast w$ so $x(w)\cap Y_{F'}$ contains $w\cap Y_{F'}$. Since $Y^{F}_{F'}=Y_{F'}\cap CF$, we have $w\cap Y^{F}_{F'}\subset w\cap Y_{F'}$. Hence by lemma 3.19, $\pi^{F}_{F'}(w)$ is uniformly coarsely coincide with $\pi_{F'}(x(w))$. This proves that the projection maps uniformly coarsely commute with the induced map $x$. 
 \end{proof} 
 
 The last ingredient is the slightly modified version of the strong bounded geodesic image(Strong BGI) for graphs, which was proved by [\cite{BHMS20}, Lemma 5.1]. 
 \begin{proposition}[Strong BGI]
 	Let $X$ be a $\delta$-hyperbolic graph and let $V$ be a nonempty subgraph of $X$. Let $L_{V}$ be an induced subgraph of $X$ such that for each vertex $x\in V$, $L_{V}\subset N_{2}(lk(x))$. Let $X_{V}$ be the induced subgraph of $X$ whose vertex set is $X^{(0)}-V^{(0)}$, and suppose $X_{V}$ is $\delta$-hyperbolic. \par 
 	Suppose that $L_{V}$ is $\delta$-hyperbolic and $(\delta,\delta)$-quasi-isometrically embedded in $X_{V}$, and let  $\pi:X_{V}\rightarrow L_{V}$ be the coarse nearest-point projection. \par 
 	Then for any $x,y\in X$ and the geodesic $\gamma$ in $X$ from $x$ to $y$, if $\gamma\cap V=\emptyset$, then $d_{L_{V}}(\pi(x),\pi(y))<C$ for some $C=C(\delta)$. 
 \end{proposition}
 Compare to the strong BGI in \cite{BHMS20}, we weaken the condition $L_{V}\subset lk(x), x\in V$  by its small neighborhood $L_{V}\subset N_{2}(lk(x)), x\in V$ so we can apply the proposition to the CHHF structures. Nevertheless, proof of the proposition requires the condition $L_{V}\subset lk(x), x\in V$ only for bounding the distance between $\pi(x),\pi(y)$ in $X$ via some vertex in $V$, which is true if we weaken the condition by its 2 neighbourhood, so the proposition can be proved by mimic the arguement in \cite{BHMS20}. \par 
 
 \begin{lemma}
 	There exists  $C=C(n, \delta)$ so that for each $F\in\mathfrak{F}-\{ X \}$ and $x,y\in Y_{F}$, if $d_{CF}(p_{F}(x),p_{F}(y))\geq C$ then any geodesic $\gamma\subset CX$ from $x$ to $y$ intersects $PF$.
\end{lemma}
\begin{proof}
	 Put $CX, PF, CF, Y_{F}, p_{F}$ in place of $X, V, L_{V}, X_{V}, \pi $ in the Strong BGI respectively. $CX$ and $CF$ are $\delta$-hyperbolic by the hyperbolicity condition of CHHF. $PF$ is non-empty since $F\sqsubsetneq X$ implies $b_{F}\in PF$. Also for each $v\in PF$, either (1) $v\in F^{\perp}$ or $v=b_{H}$ with $F\sqsubseteq H$, in which case $CF\subset N_{1}(F)\subset N_{1}(lk(v))$ or (2) $v=b_{H}$ with $F\perp H$, so $CF\subset N_{1}(F)\subset N_{2}(H)=N_{2}(lk(v))$. $Y_{F}$ is $\delta'(n, \delta)$-hyperbolic by 3.10 and $CF$ is $(\delta,\delta)$-quasi-isometrically embedded in $Y_{F}$ by hyperbolicity condition. We can take $\delta''=\delta''(n,\delta)$ large enough such that $\delta''$ works for all hyperbolicity constants. By Strong BGI, there exists $C=C(n,\delta'')$ such that if  $d_{CF}(p_{F}(x),p_{F}(y))\geq C$, then $\gamma\cap PF \neq \emptyset$ as desired.
	\par 
	
\end{proof}
	
 \par

 \subsection{Hierarchically Hyperbolicity of CHHF}
 Now we state our main theorem.
  \begin{theorem}
  	Let $(X,W,\mathfrak{F})$ be a $(n,\delta)$-CHHF triple and let $(W,\mathfrak{F})$ be the corresponding $E(n,\delta)$-proto-hierarchy structure. Then $(W,\mathfrak{F})$ is $E'(n,\delta)$-hierarchically hyperbolic space.
  \end{theorem}
\begin{proof}
  We will use the induction arguement on the complexity $n$ in the proof various items to guarantee that for each domain $F\in \mathfrak{F}-\{ X \}$, the corresponding CHHF $(F, W^{F}, \mathfrak{F}_{F})$ induces the HHS structure on $(W^{F}, \mathfrak{F}_{F})$. First we check the base case $n=1$, that is $\mathfrak{F}=\{X\}$. Since $\mathfrak{F}$ contains all the link graphs of vertices of $X$, so $\mathfrak{F}=\{ X\}$ implies that no vertex of $X$ has non-empty link graph, that is $X$ is a discrete set. Also each vertex of $X$ is itself a maximal clique subgraph of $X$ and no additional projection vertices are added, so $W=CX$ which is $\delta$-hyperbolic by hyperbolicity condition and hence $W$ is an HHS. \par 
  Now consider the proto-hierarchy structure $(W,\mathfrak{F})$ constructed in the lemma 3.17. We verify the HHS axioms for $(W,\mathfrak{F})$ with constants depending only on $n$ and $\delta$ (we mostly just check that the constants can be chosen uniformly): 
\begin{enumerate}
	\item (Hyperbolicity) Hyperbolicity is direct from the hyperbolicity condition of CHHF.
    \item (Container)
     To check the container condition, suppose that $F'\sqsubsetneq F$ and
    $ \{ H\in\mathfrak{F}_{F} \mid H\perp F'  \} \neq \emptyset$. We claim that $(F')^{\perp} \cap F$, which is a domain properly nested in $F$, is the container of $F'$. Whenever $H\in\mathfrak{F}$ such that $H\sqsubsetneq F$ and $H\perp F'$, we have $H\subset (F')^{\perp}$ and $H\subset F$, hence $H\subset (F')^{\perp}\cap F$ as desired.
    
    \item (Finite Complexity) The finite complexity on the domains is direct from the definition of the graph factor system and our definition of nesting.

     \item (Consistency) 
     Suppose $F\pitchfork F'$. Let $w\in W$ be a vertex in $W$, considered as a maximal clique of $X$. Suppose $d_{CF}(\pi_{F}(w), \rho^{F'}_{F} ) > C$, where $C$ is the constant in lemma 3.22. By lemma 3.22, the geodesic $\gamma$ in $CX$ joining the vertex in $w\cap Y_{F}$ and $PF'\cap Y_{F}$ must intersect $PF$. Let $y$ be the point on $\gamma\cap PF$ closest to $w$. Then the subgeodesic joining the vertex in $w$ and $y$ does not intersect $PF'$, so using lemma 3.22 again, we have $d_{CF'}(\pi_{F'}(w), \rho^{F}_{F'} ) <C$.
     \par 
     Now consider the consistency of nesting pairs $F\sqsubsetneq F'$. Let $w\in W$ and consider the geodesic between $w\cap Y_{F}$ and $\pi_{F'}(w)=p_{F'}(w\cap Y_{F})$ in $Y_{F'}$, say $\gamma\subset Y_{F'}$. We have two possibilities: First suppose that $\gamma \cap PF=\emptyset$. Since $F\sqsubsetneq F'$ implies $Y_{F}\subsetneq Y_{F'}$, $\gamma$ is a geodesic of $Y_{F'}$ that is entirely contained in the induced subgraph $Y_{F}$ and hence is a geodesic of $Y_{F}$. We apply the strong BGI to $\gamma$: $Y_{F}$ is hyperbolic and is an induced subgraph of $Y_{F'}$ obtained by deleting $PF\cap Y_{F'}$, $CF$ is hyperbolic, quasi-isometrically embedded subgraph of $Y_{F}$, and $CF\subset N_{2}(lk(x))$ for any $x\in PF\cap Y_{F'}$, $p_{F} : Y_{F}\rightarrow CF$ is the nearest point projection. Hence by Proposition 3.21, the assumption that $\gamma \cap (PF\cap Y_{F'})=\emptyset$ implies
     \begin{center}
     	$d_{CF}(p_{F}(w\cap Y_{F}),p_{F}(p_{F'}(w\cap Y_{F})))=d_{CF}(\pi_{F}(w),\rho^{F'}_{F}(\pi_{F'}(w))<C$.
     	
     \end{center} 
     Since both images have uniformly bounded diameter, we have uniform bound on $diam_{F}(\pi_{F}(w)\cup \rho^{F'}_{F}(\pi_{F'}(w)))$
     as desired. \par 
     Next, suppose that $\gamma$ intersects $(PF\cap Y_{F'}) \neq \emptyset$ at some point, say $v\in PF\cap Y_{F'}$. Then $CF\subset lk(v)$ and for any $u\in CF$, $d_{Y_{F'}}(w\cap Y_{F},u)\leq d_{Y_{F'}}(w\cap Y_{F},v)+1$. On the other hand, by the definition of $p_{F'}$, $d_{Y_{F'}}(w\cap Y_{F},u)\geq d_{Y_{F'}}(w\cap Y_{F}, p_{F'}(w\cap Y_{F}))-1$. Hence $d_{Y_{F'}}(w\cap Y_{F},\pi_{F'}(w))-d_{Y_{F'}}(w\cap Y_{F},v) \leq 2$. Since $v, w\cap Y_{F}, \pi_{F'}(w)$ are the points on the geodesic $\gamma$, $v$ and $p_{F'}(w\cap Y_{F})$ have distance at most 2. Since $v\in PF\cap Y_{F'}$, $p_{F'}(v)\subset \rho^{F}_{F'}$ and the projection map is uniform Lipschitz map, the distance between the projections $\rho^{F}_{F'}$ and $\pi_{F'}(w)$ is also uniformly bounded. This conclude the consistency for nesting.
     \par      
     Finally suppose that $F\sqsubseteq F'$ and $H\in\mathfrak{F}$ such that $F'\sqsubsetneq H$ or $F'\pitchfork H$ and $H \notperp F $.
     We claim that $d_{CH}(\rho^{F}_{H},\rho^{F'}_{H})$ is uniformly bounded by constant depending only on $n$ and $\delta$. We already obserbed that $F\sqsubsetneq F'$ implies $PF'\subset PF$. So we have $p_{H}(PF'\cap Y_{H})\subseteq p_{H}(PF\cap Y_{H})$. Since we defined $\rho^{F'}_{H}=p_{H}(PF'\cap Y_{H}),\ \rho^{F}_{H}=p_{H}(PF\cap Y_{H})$ in all possibilites, we are done.

     \item (Bounded Geomdesic Image)
     Suppose $F\sqsubsetneq F'$. Let $x,y\in W$ and $\gamma$ be a geodesic between $\pi_{F'}(x)$ and $\pi_{F'}(y)$ in $CF'$. \par 
     First we claim that if $\gamma$ lies sufficiently far from the set $\rho^{F}_{F'}$, then $\gamma$ is contained in $Y_{F}$. Since $\gamma\subset CF'$, we only need to check that $\gamma \cap (PF\cap CF')=\emptyset$. Note that $\rho^{F}_{F'}=p_{F'}(PF\cap Y_{F'})$. Since $F\sqsubsetneq F'$, we have $b_{F}\in CF'\cap PF$ and $p_{F'}(b_{F})=b_{F}\in\rho^{F}_{F'}$. Since both $\rho^{F}_{F'}$ and $PF\cap CF'$ have uniformly bounded diameters and have a point in common, their union also has a uniform bound on its diameter, say $E'$. If $\gamma$ is $E'$-far from $\rho^{F}_{F'}$, then $\gamma\cap (CF'\cap PF)=\emptyset$ as desired. 
     \par 
     Now we can regard $\gamma$ as a geodesic in $CF'\cap Y_{F}=Y^{F'}_{F}$. Also note that $\gamma$ lies far from the augmented subgraph $CF$ since $CF$ is contained in the 2-neighborhood of $b_{F}\in \rho^{F}_{F'}$. The induced CHHF structure of the domain $F'$ implies that $CF$ is a quasi-convex subspace of the hyperbolic space $Y^{F'}_{F}$. Hence, provided $E$ is chosen sufficiently large in terms of the hyperbolicity constants, the projection of the endpoints of $\gamma$ to $CF$ are $E$-close to each other. By the consistency for nesting, the projections of the endpoints $p_{F}(\pi_{F'}(x)), p_{F}(\pi_{F'}(y))$ are uniformly coarsely coincide with $\pi_{F}(x),\pi_{F}(y)$ respectively. By modifying the uniform constant $E$ with the consistency constant, we are done.

      \item (Partial Realization)
     Let $F_{1},...,F_{k}\in\mathfrak{F}$ be pairwise-orthogonal domains. For each $i$, let $p_{i}\in \pi_{i}(X)\cap F_{i}$. Since $p_{i}\in F_{i}$ and the domains are pairwise-orthogonal, $p_{i}\in F^{\perp}_{j}$ for each $j \neq i$, so $p_{i}$ and $p_{j}$ are adjacent for each $i \neq j$. Hence we can take a maximal clique $w$ of $X$ containing $p_{1},...,p_{k}$. Since $w$ is a maximal clique of $X$, $w$ is a vertex of $W$. For each $i$, we have $w\cap CF_{i}$ contains $p_{i}$, so $p_{i}\in \pi_{F_{i}}(w)$. \par 
     Now suppose $F\in\mathfrak{F}$ satisfies $F\pitchfork F_{i}$ or $F_{i}\sqsubsetneq F$ for some $i$. If $F_{i}\sqsubsetneq F$, then $F_{i}\subsetneq F$, so $p_{i}\in F_{i}\subset Y_{F}$. Hence $p_{F}(p_{i})\subset \pi_{F}(w)$. On the other hand, $b_{F_{i}}$ lies in $CF$ and hence in $Y_{F}$, which is adjacent to $p_{i}$. Also $b_{F_{i}}\in PF_{i}$. Hence their projections via $p_{F}$ are uniformly close, which implies that $\pi_{F}(w)$ and $\rho^{F_{i}}_{F}$ are uniformly close in $CF$.\par 
     When $F\pitchfork F_{i}$, again we have $b_{F_{i}}\in Y_{F}$ and $b_{F_{i}}$ is at most distance 4 from $w$ via any vertex of $F_{i}$ and $F_{j\neq i}$, hence their images under $p_{F}$, $\pi_{F}(w)$ and $\rho^{F_{i}}_{F}$ are uniformly close in $CF$ as desired.

     \item (Large Link)
     Let $F\in\mathfrak{F}$ and $w_{1},w_{2}\in W$ be maximal cliques in $X$. First consider the case $F=X$. We need to find $F_{i}$, $i\leq N$ with $N$ bounded linearly in $d_{CX}(w_{1},w_{2})$ such that whenever $F'\in\mathfrak{F}$ satisfies $d_{CF'}(w_{1},w_{2})>C$, $F'\sqsubseteq F_{i}$ for some $i$. \par 
     Let $\gamma$ be a geodesic in $CX$ joining the closest points in $w_{i}$.  For each vertex $x_{i}\in\gamma$, associate a domain $F_{i}$ by $F_{i}=lk(x)\cap X$ if $x\in X$ and $F_{i}=F''$ if $x=b_{F''}$. (so we take $N=d_{CX}(w_{1},w_{2}.)$ By Lemma 3.22, $d_{CF'}(w_{1},w_{2})>C$ only when the geodesic $\gamma$ intersects $PF'$. If $\gamma$ intersects $PF'$ at some point, say $x_{i}$,  then we have two possibilities: \par 
     Case 1. $x_{i}\in (F')^{\perp}$ or $x_{i}=b_{H}$ and $F'\sqsubseteq H$. In this case, since $F'\subseteq lk(x_{i})$ in both case and $F_{i}$ is the domain $lk(x_{i})\cap X$ in all possibilities, so we have $F'\sqsubseteq F_{i}$. \par
     Case 2. $x_{i}=b_{H}$ and $F'\perp H$. In this case, note that each vertices adjacent to $b_{H}$ must be contained in $(F')^{\perp}$. Let $x_{j}$ be a vertex of $\gamma$ that is adjacent to $x_{i}$. Then $x_{j}\in (F')^{\perp}$ and hence by the case 1, $F'\sqsubseteq F_{j}$ as desired. 
     \par Now suppose $F\neq X$. By induction on complexity, we know that $(W^{F}, \mathfrak{F}_{F})$ is an HHS. Note that the coordinate tuples $(\pi_{H}(w_{i}))_{H\in\mathfrak{F}}$ are consistent tuples with uniform constants by the consistency axiom. Its subtuple $(\pi_{H}(w_{i}))_{H\in\mathfrak{F}_{F}}$ is then a consistent tuple of the HHS $(W^{F}, \mathfrak{F}_{F})$. The realisation theorem implies that there exists maximal cliques $w_{i}'$ of $F$ such that the projections of $w_{i}'$ to each $CH, \ H\in\mathfrak{F}_{F}$ are uniformly close to $\pi_{H}(w_{i})$ respectively. Since $(W^{F}, \mathfrak{F}_{F})$ is HHS (with HHS constants $E'$), Large Links for $w_{1}', w_{2}'$ about the nesting maximal domain $F\in \mathfrak{F}_{F}$ implies that there exists $N=E' d_{CF}(\pi_{F}(w_{1}'),\pi_{F}(w_{2}'))+E'$ and a collection of domains $\{T_{i}\}_{i=1,...,\lfloor N \rfloor}\subseteq \mathfrak{F}_{F}-\{F \}$ such that whenever $T\in \mathfrak{F}_{F}-\{F\}$ with $d_{CT}(\pi_{T}(w_{1}'),\pi_{T}(w_{2}'))>E'$, there exists some $i$ such that $T\sqsubseteq T_{i}$. Since $\pi_{T}(w_{i}')$ are uniformly coarsely coincide with $\pi_{T}(w_{i})$ for every domain $T\in\mathfrak{F}_{F}$, by increasing the constant $E$ if necessary, we obtain Large Links for $w_{1},w_{2}$.
     
     \item (Uniqueness)
     First note that by induction on complexity, we can assume that for each non-maximal domain $F\in \mathfrak{F}- \{X\}$, the induced  $(W^{F}, \mathfrak{F}_{F})$ admits an HHS structure. \par 
     Let $w_{1},w_{2}\in W$ be two maximal cliques such that for every $F\in\mathfrak{F}$, $d_{CF}(w_{1},w_{2})<\kappa$. We need to find $\theta$  depending only on $\kappa, \delta, n$ such that $d_{W}(w_{1},w_{2})<\theta$. We will show this by using induction on $d_{CX}(w_{1},w_{2})=k$. \par 
     If $d_{CX}(w_{1},w_{2})=0$, then $w_{1}\cap w_{2}$ is non-empty. Let $x\in w_{1}\cap w_{2}$. Since both $w_{1}$, $w_{2}$ are cliques, $w_{i}'=w_{i}-{x}$ are cliques in $lk(x)$. Furthermore, $w_{i}'$ are maximals in $lk(x)$ and $x$ is a maximal clique in $(lk(x))^{\perp}$ containing $x$. 
      By the hierarchy condition, there is a map $x : W^{lk(x)} \rightarrow W$ that sends each maximal clique subgraph $w'$ of $lk(x)$ to some maximal clique subgraph $w$ of $X$ such that $w'\ast x \subseteq w$. Since $x(w_{i}')$ contains $w_{i}'\ast x=w_{i}$ and $w_{i}$ are maximal cliques, so $x(w_{i}')=w_{i}$.
     For each $F\sqsubset lk(x)$, $w_{i}\cap Y_{F}=w_{i}'\cap Y_{F}$ so $d_{CF}(w_{1}',w_{2}')=d_{CF}(w_{1},w_{2})<\kappa$. By uniqueness on the HHS $(W^{lk(x)},\mathfrak{F}_{lk(x)})$, we have $d_{W^{lk(x)}}(w_{1}',w_{2}')<\theta'$ where $\theta'$ depends on $n,\delta,\kappa$ only. Since $x :  W^{lk(x)} \rightarrow W$ is an induced map and  $x(w_{i}')=w_{i}$, we have $d_{W}(w_{1},w_{2})<\theta'$ as desired.
     \par 
     Now suppose  $d_{CX}(w_{1},w_{2})=k>0$. Let $\gamma\subset CX$ be the shortest geodesic from $w_{1}$ to $w_{2}$. Denote the end points of $\gamma$ by $x\in w_{1}$ and $y\in w_{2}$. Also denote the second vertex of $\gamma$ by $t$. We will find some maximal clique subgraph in $W$ which satisfies the assumption of the uniqueness axiom and the distance from $w_{i}$ is less than $k$. The induction on $k$ then induces the desired bound on $d_{W}(w_{1},w_{2})$.
     \par 
     Case 1. There exists some domain $F\sqsubseteq lk(x)$ such that $PF$ intersects $\gamma$ other than $x$ and $t$. In this case, let $u$ be the closest point to $y$ such that $u\in PF\cap \gamma$ for some $F\sqsubseteq lk(x)$ (so we also have $u\neq t$). First note that $u\neq b_{H}$ with $H\perp F$. To show this, suppose that $u=b_{H}$ and $H\perp F$. Then $u$ is an additional vertex so $u\neq y$ and hence $\gamma$ contains some vertex $v$ that is adjacent to $u$ and lies closer to $y$. Since $u\neq t$, $v$ have distance at least 3 from $x$. By the construction of adjacency with the projection vertices, however, $v$ must be a vertex in $F^{\perp}$ so any vertex in $F$ are adjacent to $v$ and $x$ simultaniously. Hence the distance between $x$ and $v$ are at most 2, which contradicts to the assumption that $\gamma$ is a geodesic.\par

     Now we can assume that $F\subset (lk(u)\cap lk(x))\cap X \neq \emptyset$. Denote the domain $lk(x)\cap lk(u)$ by $F'\in\mathfrak{F}$. By induction hypothesis, $(W^{F'},\mathfrak{F}_{F'})$ is an HHS. Consider the consistent tuple $\{ \pi_{H}(w_{2}) \}_{H\in\mathfrak{F}_{F'}}$. Applying the realization theorem to the tuple, we obtain some maximal clique subgraphs $w'$ of $F'$ such that for every $H\in\mathfrak{F}_{F'}$, the coordinates $\pi_{H}(w')$ and $\pi_{H}(w_{2})$ are uniformly coarsely coincide. We take some maximal clique subgraph in $lk(w')\cap lk(x)$ (which will be denoted by $F_{w'}$ for simplicity.) in the similar way.  If $F_{w'}$ is a non-empty domain, then consider the consistent tuple $\{ \pi_{H}(w_{2}) \}_{H\in\mathfrak{F}_{F_{w'}}}$ and take its realization point $\tau'$ in the induced HHS space $W^{F_{w'}}$. Since $w'\subset lk(x)$, $w'\cup x$ spans a clique subgraph in $(F_{w'})^{\perp}$. Let $z$ be some maximal clique subgraph in $F_{w'}^{\perp}$ that contains $w'\cup x$. By the hierarchy condition and lemma 3.20, $z(\tau')$( which will be denoted by $\tau$) is the maximal clique of $W$ whose coordinates are uniformly close to $w_{2}$ over the domains $H\in\mathfrak{F}_{F'}$ or $H\in\mathfrak{F}_{F_{w'}}$. If $F_{w'}=\emptyset$, then take $\tau'$ to be any maximal clique subgraph in $(F')^{{\perp}}$ containing $x$. Again the hierarchy condition produces the maximal clique subgraph $\tau=\tau'(w')\subset W$ whose coordinates are uniformly coarsely coincide over the domains $H\sqsubseteq F'$.
     
      \par

      We claim that the projections of the maximal cliques $w_{2}$, $\tau$ are uniformly coarsely coincide in entire domains $H\in\mathfrak{F}$.\par 
     If a domain $H$ is not nested in $lk(x)$, then $x\in Y_{H}$. Since  $w_{1}$ and $\tau$ both contains $x$, the coordinates $\pi_{H}(w_{1})$ and $\pi_{H}(\tau)$ shares a non-empty subset $p_{H}(x)$ in common. This shows that $\pi_{H}(w_{1})$ (in particular, $\pi_{H}(w_{2}))$) is uniformly coarsely coincide with $\pi_{H}(\tau)$ as desired.  \par 
     
     If a domain $H$ is nested in $lk(x)$ but not nested in $lk(u)$, then by our choice of $u$, the subgeodesic of $\gamma$ starting from $u$ does not intersect $PH$. Now we have two possibilities: Either $\tau\cap Y_{H}$ contains a vertex of $w'$ or $w'\subset H^{\perp}$. If $\tau\cap Y_{H}$ contains a vertex $v\in w'$, then $v$ and $u$ are adjacent and the concatenation of the edge $[v,u]$ and the subgeodesic of $\gamma$ from $u$ to $y$ forms a geodesic between $v$ and $y$. Also such geodesic cannot intersect $PH$. Hence by Lemma 3.22, the projections $p_{H}(v)$ and $p_{H}(y)$ are uniformly close and in particular, $\pi_{H}(\tau)$ and $\pi_{H}(w_{2})$ are uniformly coarsely coincide. If $w'\subset H^{\perp}$, then $H\subset lk(w')\cap lk(x)$ and by our construction, $\pi_{H}(\tau)$ is uniformly coarsely coincide with $\pi_{H}(w_{2})$.
     
     We checked that the maximal clique $\tau\in W$ satisfies the assumption of the uniqueness condition with $w_{i}, i=1,2$. Since $d_{CX}(\tau,w_{2})\leq d_{CX}(u,y)+1<k$ and $d_{CX}(\tau,w_{1})=0$, the induction arguement on $k$ implies that $\tau$ and $w_{2}$($w_{1}$) are uniformly close in $W$. In particular, $w_{1}$ and $w_{2}$ are uniformly close. 
     \par 
     Case 2. $t\in X$ and for each $F\sqsubseteq lk(x)$, $PF\cap \gamma \subseteq \{ x,t \}$. 
     \par Subcase 1. $t$ is adjacent to $x$ in $X$. \par 
     In this case, construct a maximal clique of $X$ as follows. If $x\ast t$ is not a maximal clique subgraph of $X$, then the induction arguement tells that $W^{lk(t\ast x)}$ is an HHS. Again we consider the consistent tuple $\{ \pi_{H}(w_{2}) \}_{H\in\mathfrak{F}_{lk(t\ast x)}}$ and its realization point $w'\subseteq W^{lk(t\ast x)}$. Take a maximal clique subgraph in $lk(t\ast x)^{\perp}$ containing $t\ast x$ (which is just $t\ast x$ itself) and consider the associated maximal clique graph $(t\ast x)(w')=w$. If $x\ast t$ is a maximal, then simply take $w=x\ast t$. Similar to the case 1, each domain $H\in\mathfrak{F}$ satisfies either $H\nsqsubseteq lk(x)$, in which case $w$ and $w_{1}$ shares $x\in Y_{H}$ so $w$ and $w_{i}$ are uniformly close in $CH$, or $H\sqsubseteq lk(x)$ but $H\nsqsubseteq lk(t)$, in which case the subgeodesic of $\gamma$ starting from $t$ does not intersect $PH$ and hence $w$ and $w_{i}$ are uniformly close by lemma 3.22, or $H\sqsubseteq lk(t\ast x)$ and hence the construction of $w$ directly implies $w$ and $w_{i}$ are uniformly close in $CH$. Again we find the maximal clique subgraph that satisfies the assumption of uniqueness condition with $w_{i}$ and $d_{CX}(w,w_{i})<k$, so the induction argument implies $w_{1}$ and $w_{2}$ are uniformly close in $W$ via $w$.
     \par Subcase 2. $lk(x)\cap lk(t)\neq \emptyset$ and $t$ is not adjacent to $x$ in $X$.
     \par We run the similar process: Find the maximal clique subgraph $w'$ of $lk(x)\cap lk(t)\in\mathfrak{F}$ whose coordinates are uniformly close to $w_{i}$ in various $H\in\mathfrak{F}_{lk(x)\cap lk(t)}$. Then take the maximal clique $\tau'$ (possibly empty) in $lk(w')\cap lk(x)$ whose coordinates are again coarsely coincide with $w_{2}$ over the domains nested in $lk(x)\cap lk(w')$. Take some maximal clique graph $\tau$ contained in $lk(w')^{\perp}$ that contains $\tau'\ast x$ and using the  hierarchy condition, take the maximal clique graph $w=\tau(w')\in W$. To check the uniqueness assumption between $w$ and $w_{i}$, consider the following cases. If $H\sqsubseteq lk(x)\cap lk(t)$, then by our construction of $w$, we are done. If $H\nsqsubseteq lk(x)$, then $w$ and $w_{1}$ shares $x\in Y_{H}$ and hence their projections are coarsely coincide. If $H\sqsubseteq lk(x)$ but $H\nsqsubseteq lk(t)$, then either $w'\cap Y_{H}=\emptyset$, in which case $H\sqsubseteq lk(x)\cap lk(w')$ and our construction of $\tau'\subset w$ implies the desired bound on the distance between $w$ and $w_{i}$, or there exists some vertex $v\in w'\cap Y_{H}$. In the last situation, if the concatenation of the edge $[v,t]$ and the subgeodesic of $\gamma$ from $t$ to $y$ is a geodesic in $CX$, then by lemma 3.22, we are done. If such edge path is not a geodesic between $v$ and $y$, then since $v$ is adjacent to $x$ and $\gamma$ is a geodesic, the concatenation of the edge $[x,v]$ and the geodesic between $v$ and $y$ has the same length with $\gamma$ and in particular, geodesic of $CX$ between $x$ and $y$. If $PH$ intersects the geodesic between $v$ and $y$, then it is the case 1. If $PH$ does not intersects the geodesic between $v$ and $y$, then again by lemma 3.22, the projections of $w$ and $w_{2}$ are uniformly close as desired. \par 
     Since $d_{CX}(w,w_{1})=0$, $w$ is uniformly close to $w_{1}$ in $W$. By construction, $w$ contains the points in $lk(t)$ which is adjacent to $t$ in $X$, so $w$ and $w_{2}$ satisfies the assumption of the subcase 1 above, hence $w$ (and hence $w_{1}$) and $w_{2}$ are uniformly close in $W$ as desired.
     \par 
     Subcase 3. $lk(x)\cap lk(t)=\emptyset$ and $t$ is not adjacent to $x$ in $X$.
     \par In this case, there exists $W$-adjacent maximal clique subgraphs $p, q\subseteq X$ such that $x\in p, t\in q$. Since $p$ and $q$ are $W$-adjacent, there projections are uniformly close in all domains. Also $p$ and $w_{1}$ satisfies the assumption of the uniqueness condition. For each domain $H\in\mathfrak{F}$, if $H\sqsubseteq lk(x)$, then since $lk(x)\cap lk(t)=\emptyset$, $PH$ does not intersect the subgeodesic of $\gamma$ starting from $t$ and hence by 3.22, $q$ and $w_{2}$ are uniformly close in $CH$ and so does for $p$ and $w_{1}$. If $H\nsqsubseteq lk(x)$, then $p$ and $w_{1}$ shares $x\in Y_{H}$ so $p$ and $w_{1}$ are uniformly close in $CH$. Since $d_{CX}(p,w_{1})=0$, $p$ and $w_{1}$ are uniformly close in $W$. Since $p$ and $q$ have coarsely coincide coordinates, $q$ and $w_{2}$ satisfies the assumptions of uniqueness condition and $d_{CX}(q, w_{2})<k$, so $q$ and $w_{2}$ are uniformly close in $W$. Since $p$ and $q$ are $W$-adjacent, $w_{1}$ and $w_{2}$ are uniformly close in $W$ as desired.
     \par Case 3. $t\notin X$
     \par Let $t=b_{F}$ and denote the vertex of $\gamma$ next to $t$ by $u$.
     \par Subcase 1. $u=y$. \par 
      First we will find some maximal clique subgraphs in $F$ whose coordinates over $H\in\mathfrak{F}_{F}$ are uniformly close to $w_{i}$ and contains $x$ and $y$ respectively. If $x$ is a maximal clique in $F$, then take $w_{1}'=x$. If $lk(x)\cap F\neq \emptyset$, then consider the consistent tuple $\{ \pi_{H}(w_{1}) \}_{H\in\mathfrak{F}_{lk(x)\cap F}}$ and its realisation point $z_{1}$ in $lk(x)\cap F$. Take some maximal clique subgraph $\tau \subset (lk(x)\cap F)^{\perp}\cap F$ containing $x$ and denote $\tau(z_{1})=w_{1}'$. Similarily take $w_{2}'$ by considering consistent tuple and its realisation in $lk(y)\cap F$ and extent to $W^{F}$ via the maximal clique of $(lk(y)\cap F)^{\perp}$ containing $y$. Then $w_{1}'$ and $w_{1}$ (respectively for $w_{2}'$ and $w_{2}$) have uniformly close coordinates over $H\in \mathfrak{F}_{F}$. To check this, consider any domain $H\in\mathfrak{F}_{F}$. If $H\sqsubseteq lk(x)$ ($H\sqsubseteq lk(y)$ respectively), then by our construction, the coordinates of $w_{1}'$ and $w_{1}$ (for $w_{2}'$ and $w_{2}$ respectively,) are uniformly close in $CH$. If $H\nsqsubseteq lk(x)$ ($H\nsqsubseteq lk(y)$ respectively), then $w_{1}'$ and $w_{1}$ shares a point $x\in Y_{H}$ ($y\in Y_{H}$ for $w_{2}'$ and $w_{2}$) and hence again $w_{1}'$ and $w_{1}$ are uniformly close in $CH$. \par 
      Since the coordinates of $w_{i}'$ are uniformly close to the coordinates of $w_{i}$, we have that $w_{i}'$ are the points of $W^{F}$ satisfying the assumption of the uniqueness condition. By induction hypothesis, $(W^{F},\mathfrak{F}_{F})$ is an HHS and by the uniqueness, $w_{1}'$ is uniformly close to $w_{2}'$ in $W^{F}$. By hierarchy condition of the CHHF, taking any maximal clique $\tau'\subset F^{\perp}$, there exists an induced map $\tau' : W^{F}\rightarrow W$. Since $\tau'$ is an induced map, $\tau'(w_{1}')=\tau_{1}$ and $\tau'(w_{2}')=\tau_{2}$ are uniformly close in $W$. We claim that $\tau_{i}$ are uniformly close in $W$ to $w_{i}$ respectively, which completes the subcase 1. \par 
     We first note that, since $x\in \tau_{1}$ ($y\in \tau_{2}$ respectively), we have $d_{CX}(\tau_{i},w_{i})=0$. Hence if the coordinates of $\tau_{i}$ and $w_{i}$ are uniformly close in every domains $H\in\mathfrak{F}$, then the induction arguement implies  $\tau_{i}$ and $w_{i}$ are uniformly close in $W$ as desired. Note that the coordinates of $w_{1}$ and $w_{2}$ are uniformly close by assumption, and the coordinates of $\tau_{1}$ and $\tau_{2}$ are uniformly close by the uniform lipschitz property of the projection maps. Hence we only need to check the statement for one pair. If there exists some domain $H$ which satisfies $H\sqsubseteq lk(x)$ and $H\sqsubseteq lk(y)$, then by taking $v\in H$ we can replace $\gamma$ by the geodesic  $x$-$v$-$y$ whose second vertex is an element of $X$, which was done in the case 1 and 2. So we can assume that for each domain $H$, either $H\nsqsubseteq lk(x)$ or $H\nsqsubseteq lk(y)$. If $H\nsqsubseteq lk(x)$, then $\tau_{1}$ and $w_{1}$ shares a point $x\in Y_{H}$, so $\tau_{1}$ and $w_{1}$ are uniformly close in $CH$. If $H\nsqsubseteq lk(y)$, then $\tau_{2}$ and $w_{2}$ are uniformly close in $CH$ in the same reason.
     \par 
     Subcase 2. $u\neq y$. In this case denote the vertex of $\gamma$ that appears after $u$ by $v$. \par 
     Again we will find a maximal clique graph containing $u$ and whose coordinates are uniformly close to $w_{i}$. First we find a   maximal clique in $W^{F}$ with desired coordinates. \par 
     If $lk(v)\cap F=\emptyset$, then consider $lk(u)\cap F$. If $lk(u)\cap F=\emptyset$, then $u$ is a maximal clique subgraph of $F$ and we simply take $z=u$. If $lk(u)\cap F\neq \emptyset$, then consider the consistent tuple $\{ \pi_{H}(w_{2}) \}_{H\in\mathfrak{F}_{lk(u)\cap F}}$ and its realisation point $z'\in W^{lk(u)\cap F}$, considered as a maximal subgraph of $lk(u)\cap F$. Take a maximal clique subgraph $\tau$ in $(lk(u)\cap F)^{\perp}\cap F$ that contains $u$ and using the hierarchy condition of the CHHF $(F,W^{F})$, take $\tau(z')=z$ which is a maximal clique subgraph of $F$. Also take $w_{1}'\in W^{F}$ constructed in the subcase 1, which is the maximal clique graph of $F$ that contains $x$ and have coordinates uniformly close to $w_{1}$. We claim that $w_{1}'$ and $z$ have uniformly coarsely coincide coordinates over $H\in\mathfrak{F}_{F}$. \par 
     Let $H\in\mathfrak{F}_{F}$. Note that $PH\cap \gamma \subseteq \{ x, t, u\}$. To see this, first note that since $\gamma$ is a geodesic and $H\in lk(t)$, the only possible case is $v\in PH$. If $v\in PH$, then either $H\sqsubseteq lk(v)$ and hence $\emptyset \neq H\subseteq F\cap lk(v)$ which contradicts to our assumption, or $PH\cap \gamma$ contains a vertex of $\gamma$ with distance at least 3 from $t$, which can be joined by an edge path of length 2 with $t$ via some vertex of $H$, contradicts to the assumption that $\gamma$ is a geodesic.\par  
     If $H\sqsubseteq lk(u)$, then $H\sqsubseteq lk(u)\cap F$. By the construction of $z$, the coordinates of $z$ in $CH$ is uniformly close to $\pi_{H}(w_{2})$. If $H\nsqsubseteq lk(u)$, then the subgeodesic of $\gamma$ from $u$ to $y$ misses $PH$, hence by lemma 3.22, the coordinates of $z$ and $w_{2}$ are uniformly close. \par 
     Since the coordinates of $w_{1}'$ are uniformly coarsely coincide with the coordinates of $w_{i}$ over the domains $H\in\mathfrak{F}_{F}$, $w_{1}'$ and $z$ have uniformly close coordinates as desired.  \par 
     In the case of $lk(v)\cap F\neq \emptyset$, we run the similar process to find some maximal cliques whose coordinates are uniformly close to $w_{1}'$ over the domains $H\in\mathfrak{F}_{F}$. Consider the consistent tuple $\{ \pi_{H}(w_{2}) \}_{H\in\mathfrak{F}_{lk(v)\cap F}}$ and its realisation point $z'\in W^{lk(v)\cap F}$. If $z'$ is a maximal clique graph in $F$, then take $z=z'$. If $lk(z')\cap F\neq \emptyset$, then again consider the consistent tuple $\{ \pi_{H}(w_{2}) \}_{H\in\mathfrak{F}_{lk(z')\cap F}}$ and its realisation point $z''\in W^{lk(z')\cap F}$. Take some maximal clique subgraph $\tau$ of $(lk(z')\cap F)^{\perp}\cap F$ that contains $z'$ and let $z=\tau(z'')$. We claim that $z$ is desired clique subgraph. \par 
     Let $H\in\mathfrak{F}_{F}$. Also note that $PH\cap \gamma\subseteq \{ x, t, u, v \}$. \par 
     If $H\sqsubseteq lk(v)$, then by the construction of $z'$, the coordinates of $z$ in $CH$ are uniformly close to $\pi_{H}(w_{2})$. If $H\nsqsubseteq lk(v)$, then either $H\nsqsubseteq lk(z')$ or $H\sqsubseteq lk(z')$. In the first case, for $u'\in z'\cap Y_{H}$, the geodesic from $u'$ to $y$ obtained by concatenating the subgeodesic of $\gamma$ from $v$ to $y$ and the edges between $u'$ and $v$ misses $PH$. Hence by lemma 3.22, the coordinates of $z$ and $w_{2}$ are uniformly close. In the second case, $H\sqsubseteq lk(z')\cap F$ and by our construction of $z''$, $\pi_{H}(z)$ is uniformly close to $\pi_{H}(w_{2})$ as desired. Also $z$ and $w_{1}'$ have the coarsely coincide coordinates. \par 
     We obtain $z,w_{1}'\in W^{F}$ whose coordinates are uniformly close over all domains $H\in\mathfrak{F}_{F}$, hence using the uniqueness axiom of the HHS $(W^{F},\mathfrak{F}_{F})$, $z$ and $w_{1}'$ are uniformly close in $W^{F}$. By the hierarchy condition, we can find some maximal clique subgraphs $\tau_{1},w\subset X$ that are uniformly close in $W$ and $w_{1}'\subset \tau_{1}, z\subset w$. Since $d_{W}(w_{1},\tau_{1})=0$ and $d_{W}(w,w_{2})<k$, it only remains to show that the coordinates of $w$ or $\tau_{1}$ are uniformly close to the coordinates of $w_{i}$ over every domains $H\in\mathfrak{F}$. \par 
     If $H \sqsubseteq F$, then we are done. If $H\nsqsubseteq F$, then either $H\nsqsubseteq lk(x)$, in which case $\tau_{1}$ and $w_{1}$ shares a vertex $x\in Y_{H}$ and hence $\tau_{1}$ is close to $w_{i}$ in $CH$, or $H\sqsubseteq lk(x)$ and hence the subgeodesic of $\gamma$ from $u\in w$ to $y$ does not intersect $PH$ so lemma 3.22 implies $w$ is uniformly close to $w_{i}$ in $CH$ as desired.  
     
\end{enumerate}
\end{proof}

 \section{Factor system of crossing graph}
 
 The advantage of the CHHF is that the CHHF machinery covers the most well-studied curve graph with HHS structure, the curve graph of finite type surfaces and the crossing graph of CAT(0) cube complex. Also the HHS structure induced from CHHF curve graph admits the clean container condition and the wedge condition, which was found to be requirement to recover the CHHS from the HHS, so we believe that the CHHF curve graph is useful to study the connection between the HHS and combinatorial HHS theory.
 As an application, we extend the factor system machinery, developed in \cite{BHS17} for the CAT(0) cube complex, to the quasi-median graph. 
 
 \subsection{Background on Quasi-median graph}
 
 In this section we shall review some basic notions and results for the quasi-median graphs. Most of the definitions and results in this section were introduced by Genevois in \cite{Gen17}. 
 \par 
 
 \begin{definition}
 	A $k$-$quasi$-$median$ of points $x_{1},x_{2},x_{3}$ of a graph $X$ is a triple $(y_{1},y_{2},y_{3})$ of vertices of $X$ satisfying that 
 	\begin{enumerate}
 		\item $y_{i}$ and $y_{j}$ lie on a geodesic between $x_{i}$ and $x_{j}$ for any $i\neq j$;
 		\item $k=d_{X}(y_{1},y_{2})=d_{X}(y_{1},y_{3})=d_{X}(y_{2},y_{3})$; and
 		\item $k$ is as small as possible subject to 1 and 2.
 	\end{enumerate}
 A triple  $(y_{1},y_{2},y_{3})$ is a $quasi$-$median$ of  $x_{1},x_{2},x_{3}$ if it is a $k$-quasi-median for some $k$. A $0$-quasi-median is called a $median$.
 \par 
 A $quasi$-$median$ $graph$ is a graph $X$ satisfying that
 \begin{enumerate}
 	\item every triple of $X$ has a unique quasi-median,
 	\item does not contain $K^{-}_4$ as induced subgraph,
 	\item each isometrically embedded cycle $C_{6}$ has a 3-cube as its convex hull in $X$.
 \end{enumerate}

 \end{definition}
If $X$ is a quasi-median graph and every triple has $0$-quasi-median, then $X$ is a median graph. Genevois showed that the quasi-median graph is indeed the generalization of a median graph, admitting various cubical-like structure (see \cite{Gen17} for more details). Most interestingly, the geometry of quasi-median graph can be completely reformulated using the combinatoric data of its hyperplanes.
\par 
\begin{definition}
	Let $X$ be a quasi-median graph. Let $\sim$ be the equivalence relation on the edge set of $X$ defined as follows; two edges $e,f$ of $X$ are related (we call such relation '$parallel$') if either $e$ and $f$ are two sides of a triangle or two opposite sides of an induced square. A $hyperplane$ $H$ is an equivalence class $[e]$ for some edge $e$ (we call $H$ is dual to $e$ in this case). \par 
	Let $H$ be a hyperplane dual to an edge $e$. The $carrier$ of $H$, denoted by $N(H)$, is an induced subgraph of $X$ spanned by $[e]$. We call each connected component of $N(H)\setminus J$ a $fibre$ of $H$ or the $combinatorial$ $hyperplane$ of $H$, where $J$ is the union of the interiors of all the edges in $[e]$. Each hyperplane separates $X$ into at least two components, which are called the $sectors$ delimited by the hyperplane $H$. If subcomplexes $Y_{1},Y_{2}\subset X$ are contained in two different sectors delimited by a hyperplane $H$, then we say $H$ $separates$ $Y_{1}$ and $Y_{2}$.   \par 
	We say two hyperplanes $H_{1},H_{2}$ $cross$ if $H_{i}$ are dual to the edges $e_{i}$ and there exist edges $e_{i}'\in [e_{i}]$ that form a corner of some induced square of $X$. If two hyperplanes have intersecting carrier but not cross, then we call such hyperplanes $osculate$.
\end{definition}

 Before considering the combinatoric and geometric properties of hyperplanes in $X$, we define a quasi-median generalization of the convex subcomplex, the  $gated \ subgraph$.

 \begin{definition}
 	Let $X$ be a graph and $Y\subset X$ a subgraph. Fixing a vertex $x\in X$, we say that a vertex $y\in Y$ is a $gate$ for $x$ if, for every $z\in Y$, there exists a geodesic between $x$ and $z$ passing through $y$. If any vertex of $X$ admits a gate in $Y$, then  $Y$ is called a $gated \ subgraph$ of $X$. Let $Y\subset X$ be a gated subgraph. The map $\mathfrak{g} : X \rightarrow Y$ that sends each vertex of $X$ to its gate in $Y$ is called the $projection$ onto $Y$. Let $S\subset V(X)$ be a set of vertices. The $gated$ $hull$ of $S$ is the smallest gated subgraph of $X$ containing $S$. We often call the gated hull of the projection image of a subgraph $Y_{1}\subset X$ into a gated subgraph $Y_{2}$ by the $gated$ $projection$ of $Y_{1}$ to $Y_{2}$.
 \end{definition}
 
 There are useful facts about gated subgraphs of quasi-median graphs and hyperplane combinatorics. (see section 2 of \cite{Gen17} for more details.)
 
 \begin{proposition}
 	Let $X$ be a quasi-median graph and $Y\subset X$ a gated subgraph. Then
 	\begin{enumerate}
 		\item Each gated subgraphs $Y$ is itself a quasi-median graph.
 		\item Every maximal clique subgraphs are gated. The cartesian product of maximal clique subgraphs are gated. (we call such subgraph the $prism$.)
 		\item For each hyperplane $H$ of $X$, the carrier of $H$, the combinatorial hyperplanes of $H$ and the sectors delimited by $H$ are gated subgraph of $X$.
 		\item For $x,y\in X$, the hyperplanes separating the gates of $x$ and $y$ in $Y$ are precisely the hyperplanes separating $x$ and $y$ and intersect $Y$. In particular, if $Y_{1},Y_{2}$ are gated subgraphs of $X$, then the hyperplanes intersecting $\mathfrak{g}(Y_{1})$ are precisely the hyperplanes intersecting both $Y_{1}$ and $Y_{2}$.
 		\item Let $S\subset X$. The hyperplanes of the gated hull $Y$ of $S$ are precisely the restrictions to $Y$ of the hyperplanes of $X$ separating two vertices of $S$. Also two hyperplanes of $Y$ cross(contact) if and only if their extensions cross(contact) in $X$.
 		\item $Y$ contains its triangles. That is, if $C$ is a clique subgraph of $X$ and contains an egde of $Y$, then $C\subset Y$.
        \item A path $\gamma\subset X$ is a geodesic if and only if $\gamma$ intersects any hyperplane at most once.
        \item The hyperplane carrier $N(H)$ of $H$ is isomorphic to the product $C\times J_{H}$ where $C$ is the maximal clique graph dual to $H$ and $J_{H}$ is a combinatorial hyperplane of $H$.
 	\end{enumerate}
 \end{proposition}
 
\begin{definition}
	A $crossing$ $graph$ of a quasi-median graph $X$, denoted by $\Delta X$, is a graph whose vertex set consists of all hyperplanes of $X$ and two distinct hyperplanes are adjacent when they are cross. Similarily, a $contact$ $graph$ of $X$, denoted by $CX$, is a graph whose vertex set is the set of all hyperplanes of $X$ and two distinct hyperplanes are adjacent when they are either cross or osculate.
\end{definition}

 In the case of median graph, the contact and crossing graph are analogous to the curve graph of the surface, they are gromov-hyperbolic and have hierarchy path which plays crucial roles to define the HHS structure on the CAT(0) cube complex with factor system. The similar result was developed by Valiunas in \cite{Val18}.

\begin{proposition}
	Let $X$ be a quasi-median graph. Then the contact graph (crossing graph, if it is connected) of $X$ is quasi-tree.
\end{proposition}
 
 The link graph of vertices and their intersection in the crossing graph have useful geometric interpretation; the crossing graphs of domains of the hyperclosure (factor system) of $X$, observed in \cite{BHS17},\cite{HS20} in the case of median graph. To see this, the key ingredient is the Sageev's construction \cite{Sa95}, that the set of hyperplanes of a median graph $X$ defines a space with wall whose cubulation isometrically realize $X$. Similar counterpart exists for the quasi-median graphs, the space with partition and its quasi-cubulation which relates the combinatorics of hyperplanes and the geometry of $X$. We omit the detailed discussion, see section 2.4 of \cite{Gen17} for more details. 
 
 \begin{proposition}[Proposition 2.63, \cite{Gen17}]
 	Let $X$ be a quasi-median graph. The set of hyperplanes of $X$ defines a space with partition. Furthermore, the quasi-cubulation of $X$, viewed as a space with partitions, is isometric to $X$.
 \end{proposition} 
 
 The proposition implies that the quasi-median graphs with the same combinatoric data of hyperplanes are isometric. In particular, two gated subgraphs of $X$ with the same set of hyperplanes are isometric. It allows to associate to each subgraph of the crossing graph  some gated subgraphs, defined up to $parallel$. 

 \begin{definition}
 	We say two gated subgraphs $Y_{1},Y_{2}\subseteq X$ are $parallel$ if for each hyperplane $H$, we have $H\cap Y_{1}\neq \emptyset$ if and only if $H\cap Y_{2}\neq \emptyset$.
 \end{definition}
 
 We note that if $Y_{1},Y_{2}$ are parallel gated subgraphs of $X$, then the gated projection of $Y_{1}$ into $Y_{2}$ is $Y_{2}$.
 \par 
 
 In the case of the median graph, the cubulation of the space with wall is constructed from the ultrafilters, considered as a map that assigns to each wall a half space in comparable way. The quasi-cubulation defined in exactly the same way; it is constructed from the choice functions which we call the $orientation$ $maps$, pick a sector from each partition in comparable way. (roughly, no two choice of sectors are disjoint.) The observation that the orientations for crossing hyperplanes never incomparable, hence the orientation map restricted on the pairwisely crossing collections of hyperplanes devided into a product. The following can be easily checked by mimic the proof of lemma 2.5 of \cite{CS11}.
 
 \begin{proposition}
 	Let $X$ be a quasi-median graph and $Y_{1},Y_{2}\subset X$ are gated subgraphs. If the set of hyperplanes crossing $Y_{1}$ and the set of hyperplanes crossing $Y_{2}$ are pairwisely crossing, then there is an isometrically embedded copy of $Y_{1}\times Y_{2}$ in $X$ which contains some parallel copies of $Y_{1}$ and $Y_{2}$ .
 \end{proposition}

 \subsection{CHHF axioms for crossing graphs}

 In this section, we will show that if a crossing graph $\Delta X$ of a quasi-median graph $X$ admits a graph factor system, then there exists a CHHF $(\Delta X, W, \mathfrak{F})$ such that $W$ is quasi-isometric to $X$. In particular, $X$ admits an HHS structure.
 
 \par 
 First we associate to each domain of the graph factor system a gated subgraph of $X$. Suppose $\Delta X$ has a graph factor system $\mathfrak{F}'$. Since $\mathfrak{F}'$ contains all the link graphs of vertices of $\Delta X $ and their intersections, we can replace $\mathfrak{F}'$ by the minimal system which consists of the link graphs of vertices of  $\Delta X $ and their intersections. \par 
 
 For each vertex $H\in\Delta X$, $H$ is a hyperplane of $X$ and the link graph of $H$ consists of those hyperplanes that crosses $H$. First we claim that the link graph $lk(H)\subset \Delta X$ is precisely the crossing graph of any combinatorial hyperplane of $H$. Note that the carrier $N(H)$ of $H$ is isomorphic to the product $C\times J_{H}$. If $H'$ is a hyperplane of $J_{H}$, then $H'$ is dual to some edge $e'$ contained in $J_{H}$ and $C\times e'$ contains an induced square, which implies that $H'$ crosses $H$. On the other hand, if $H'$ crosses $H$, then there exists edges $e,e'$ dual to $H,H'$ respectively such that $e,e'$ forms a corner of some induced square, whose opposite side of $e$ is dual to $H$. Since the gated subgraph $N_{H}$ is an induced subgraph of $X$ spaned by all parallel edges of $e$, the whole square is contained in some product $C\times C'\subset C\times J_{H}$ where $C'$ is the clique graph containing every triangles with one edge $e'$. Hence $e'$ must be contained in some fibre $J_{H}$ of $H$ and in partiular, $H'$ is a hyperplane of $J_{H}$. (indeed the prism $C\times C'$ shows that $H'$ is a hyperplane of every combinatorial hyperplanes of $H$.) We have checked that the set of hyperplanes that crosses $H$ is precisely the set of hyperplanes of $J_{H}$, hence the link graph $lk(H)$ and the crossing graph of each combinatorial hyperplanes $\Delta J_{H}$ have the same vertex set. Since $J_{H}$ is gated subgraph, the hyperplanes of $J_{H}$ crosses if and only if their extension crosses in $X$, so the adjacency of $lk(H)$ precisely consists of the crossing pairs of hyperplanes of $J_{H}$. In other words, $lk(H)=\Delta J_{H}$, $J_{H}$ considered as a quasi-median graph. \par 
 Now consider the intersection of link graphs. We only consider the intersection of two link graph of vertices, $lk(H_{1})\cap lk(H_{2})$. We claim that the intersection of two link graphs $lk(H)\cap lk(H')$ is precisely the crossing graph of the gated projection of one combinatorial hyperplane to the other. Note that the intersection $lk(H)\cap lk(H')$ consists of precisely the collection of hyperplanes that crosses both hyperplanes $H$ and $H'$, that is, the hyperplanes that intersects both combinatorial hyperplanes $J_{H}$ and $J_{H'}$. By (4) of proposition 4.4, the set of hyperplanes intersecting both $J_{H},J_{H'}$ and the set of hyperplanes intersecting the gate projection of $J_{H}$ to $J_{H'}$ are the same. Again, since the gate projection of some gated subgraph into another gated subgraph is gated, the edges in the crossing graph of the projection and the edges of $lk(H)\cap lk(H')$ are the same. \par 
 Using the same arguement inductively to general intersections we obtain the following observation. If the crossing graph $\Delta X$ admits a graph factor system, then $\Delta$ admits a graph factor system $\mathfrak{F}$ which consists of the crossing graphs of the system of gated subgraphs of $X$ obtained by taking every combinatorial hyperplanes their gated projections. (such system of subgraphs are often called the $factor$ $system$ of $X$.)\par 
 
 Now we define $\Delta X$-graph $W$ in the following way. Note that there exists a one-to-one correspondence between the maximal collection of pairwisely crossing hyperplanes and the maximal prism of $X$. This allows us to consider $W$ as the set of maximal prisms of $X$. We declare that two maximal prisms $w,w'$ of $W$ are adjacent in $W$ if $w\cap w'\neq \emptyset $ considered as the subgraphs of $X$. 
\begin{lemma}
	Let $X$ be a quasi-median graph. Let $\Delta X$ be a crossing graph of $X$ which admits a graph factor system. let $\mathfrak{F}$ be the graph factor system of $\Delta X$ constructed above. Let $W$ be the $\Delta X$-graph as above. Then the triple $(\Delta X, W, \mathfrak{F} )$ is a CHHF, in particular $(W, \mathfrak{F})$ admits an HHS structure.
\end{lemma}
 \begin{proof}
 First we check the hyperbolicity condition for the triple. To show this, first we consider the augmented graphs for each domains. If two hyperplanes $H, H'$ of $\Delta X$ are adjacent in $C(\Delta X)$ but not in $\Delta X$, then there exists maximal prisms $P, P'\subseteq X$ such that $H$ is dual to some edge of $P$, $H'$ is dual to some edge of $P'$ and $P$ is adjacent to $P'$ in $W$, that is, $P\cap P'\neq \emptyset$. Since $P$ contains an edge dual to $H$ (similarily for $P',H'$), we have $P\subseteq N(H)$ and $P'\subseteq N(H')$, so $\emptyset \neq P\cap P' \subseteq N(H)\cap N(H')$ and in particular, $H$ and $H'$ osculate. This shows that the osculating pairs are the only candidate of the additionally adjacent paris in $C(\Delta X)$. Indeed, whenever two hyperplane $H, H'$ osculate, we can take a vertex $v\in N(H)\cap N(H')$. Let $P$ be a maximal prism of $X$ containing $v$ and contains an edge dual to $H$ (similarily for $P'$). Since $H$ and $H'$ are osculate, the edges dual to $H$ and $H'$ respectively, containing $v$ in their end points, cannot be contained in a single prism and hence $P$ and $P'$ can be chosen distinctly. Since $P\cap P'$ contains $v$ by our construction, $P$ is adjacent to $P'$ in $W$ and hence $H$ and $H'$ are also adjacent in  $C(\Delta X)$. In other words, the induced subgraph of $C(\Delta X)$ spanned by $\Delta X$ is the contact graph $CX$ of $X$. Hence $C(\Delta X)$ is the contact graph of $X$ together with the collection of projection vertices $b_{J}$. This graph is nothing but the factored contact graph developed in \cite{BHS17} for the cubical cases, and using the hierarchy path on contact graph and mimic the arguements in \cite{BHS17},\cite{Val18}, we can check that the factored contact graph $C(\Delta X)$ is uniformly hyperbolic. (since the proof has no different, we omit here. The details of proof of hyperbolicity of $C(\Delta X)$ is attached in the last section.) \par 
 The same arguement shows that each augmented graphs $CJ$, $J\in\mathfrak{F}$, is the factored contact graph of the quasi-median graph $J$ and hence uniformly hyperbolic, which implies first part of the hyperbolicity condition of CHHF.
 \par 
 To check another half of the hyperbolicity condition, we need to show that each augmented subgraph $CJ$ admits a quasi-isometry embedding into the complement graph $Y_{J}$. Since $CJ$ is an induced subgraph of $Y_{J}$, it is enough to show that there exists a Lipschitz projection $\psi : Y_{J} \rightarrow CJ$ which is identity on $CJ$. We define $\psi$ in the following way; for each hyperplane vertex $H\in Y_{J}\setminus CJ$, let $g$ be the gated projection image of any combinatorial hyperplane $J_{H}$ of $H$ into $J$. If $g$ contains an edge, then $g$ has non-empty contact graph and by our construction $g\in \mathfrak{F}$. Also $g\subset J$ tells that $g\sqsubsetneq J$, so $b_{g}\in CJ$. We take $\psi (H)=b_{g}$ in this case. If $g$ is a vertex, then take $\psi (H)=H'$ for some hyperplane $H'$ of $J$ that contains $g$ in its carrier. For a non-hyperplane vertex $b_{J'}\in Y_{J}\setminus CJ$, $J'\in\mathfrak{F}$, consider the associated gated projection $g'$ of $J'$ onto $J$ and defines $\psi(b_{J'})$ in the same way. Note that when the gated projection $g$ into $J$ is a vertex, the hyperplanes containing the vertex $g$ are pairwisely contact and hence the map $\psi$ is coarsely well-defined map. Also note that in both cases, for each $v\in Y_{J}\setminus CJ$, any hyperplane of $J$ containing the vertex of $g(v)$ in its carrier lies at most distance 2 from $\psi(v)$. \par 
 To check that $\psi$ is the Lipschitz map, we only need to show that for each edge $e\subset Y_{J}$, $\psi(e)$ is uniformly bounded set. 
 If $e\subset CJ$, then $\psi(e)=e$ and we are done. If $e\cap CJ=\{ u\}$ and $e\cap (Y_{J}\setminus CJ)=\{ v \}$, then either $u$ is a hyperplane vertex that crosses both $J$ and the gated subgraph $v$ or $u=b_{J'}$, $J'\sqsubsetneq J$. In the first case, the gated projection $g$ of $v$ into $J$ is not a vertex (since some hyperplane $u$ separates some points in $g$.) and $u$ is a hyperplane of $g$. Hence $\psi(v)=b_{g}$ is adjacent to $u=\psi(u)$ as desired. If $u=b_{J'}$, $J'\sqsubsetneq J$, then since $v$ is adjacent to $b_{J'}$, $v\in CJ$ which violates our assumption that $v\notin CJ$. If $e\subset (Y_{J}\setminus CJ)$, then the gated subgraphs (either some combinatorial hyperplane or a gated subgraph contained in the factor system of $X$) of the two end points of $e$ must intersect, hence whose gated projections into $J$ also intersect. As we already observed, for some hyperplane vertex $H\in CJ$ containing the projected intersection point, the images under $\psi$ of the end points of $e$ lies at most distance 2 from $H$ and $\psi(e)$ is uniformly bounded as desired. 
 \par
 To check the hierarchy condition, for each domain $J\in\mathfrak{F}$ we first associate the $J$-graph $W^{J}$. Note that each $J\in\mathfrak{F}$ is a gated subgraph of $X$ and hence is quasi-median on its on right, so we simply associate $W^{J}$ by the graph of maximal prisms of $J$ with adjacency by nonempty intersection.

 The orthogonality relation between domains implies the corresponding gated subgraphs have pairwise crossing hyperplanes, hence by 4.9, there is an isometrically embedded copy of the products of the domains in $X$. Further the container $J^{\perp}$ consists of the maximal collection of such hyperplanes, implying that $J\times J^{\perp}$ is the maximal product subspace containing all the parallels of $J$. For each maximal clique $c\subset J^{\perp}$, $c$ corresponds to a maximal prism of $J^{\perp}$ considered as some representative gated subgraph. We define the map $c : W^{J} \rightarrow W$ by sending each maximal prism $p\subset J$ to some maximal prism of $X$ containing $p\times c$. The map $c$ is indeed the induced map. If the prisms $p,p'\subset J$ intersect in $J$, then the prisms $c(p), c(p')$ also intersect in $X$. On the other hand, if the prisms $p,p'$ are disjoint, then there exists some hyperplane $H$ of $J$ that separates $p$ and $p'$. We claim that the extension of $H$ to $X$ separates the prisms $c(p)$ and $c(p')$ in $X$. Suppose not, then we have two possibilites: either two prisms are in the same sector delimited by $H$ or $H$ separates some points in some prism, say $c(p)$. First case is impossible, since otherwise the subprisms $p,p'$ are contained in the same sector delimited by $H$ and hence $H$ cannot separates $p$ and $p'$. If $H$ separates some points in $c(p)$, then $c(p)$ contains some edges dual to $H$. Furthermore, the edge of $c(p)$ dual to $H$ is not contained in $p$, since otherwise $H$ separates some points of $p$, which is impossible by our choice of $H$. Let $h$ be the clique subgraph of $c(p)$ that is dual to $H$. Since $c(p)$ is a prism, $c(p)\subset N(H)\cong h\times J_{H}$ for some combinatorial hyperplane $J_{H}$ of $H$. Since $p\subset J_{H}$ and $H$ is a hyperplane of $J$, we can take some clique $h'\subset J$ parallel to $h$ which forms a corner of induced square with $p$, hence we have $h'\times p\subset J$, the prism of $J$ that properly contains $p$. Since we assumped that $p$ is a maximal prism of $J$, this is impossible. \par 
 Also the construction guarantees that $C_{J}(J')=C(J')$ for each $J'\sqsubseteq J$ since $C_{J}(J')$ is nothing but the factored contact graph of $J'$ with the same factor system $\mathfrak{F}_{J'}$.  
 \end{proof}
 Note that if $\Delta X$ admits a graph factor system, then there is a bound on the maximal number of pairwise-crossing hyperplanes and hence bound on the cubial dimension of $X$, so each maximal prism of $X$ have uniformly bounded diameter. Hence whenever $\Delta X$ admits a graph factor system, $X$ and $W$ is quasi-isometric. Also $X$ is an HHS. \par 
 Since the median graph is itself a quasi-median graph and is 1-skeleton of CAT(0) cube complexes, we provided a combinatorial re-proof
 of the factor system criterion for the CAT(0) cube complex.
 \begin{corollary}
 	Let $X$ be a CAT(0) cube complex, $\Delta X$ be its crossing graph. If $\Delta X$ admits a graph factor system, then $X$ is an HHS. 
 \end{corollary}
 \begin{proof}
 	We note that in the proof of 4.10, the existence of a graph factor system in $\Delta X$ and the existence of a factor system in $X$ are equivalent. If $\Delta X$ admits a graph factor system, then by 4.10, the $\Delta X$-graph $W$, considered as the graph whose vertex set consists of the maximal cubes of $X$ and edges defined from the intersection of cubes, admits an HHS structure. The map $\phi : W \rightarrow X$ that sends each vertex of $W$, considered as a maximal cube, to its midpoint is a Lipschitz map, Since the existence of the graph factor system guarantees that the dimension of maximal cubes are bounded, and hence the distance between the midpoints of two maximal cubes with common faces is uniformly bounded in terms of the diameter of the maximal dimensional cube. Also
 	its quasi-inverse map $\psi : X \rightarrow W$ that assigns to each point of $X$ any maximal cube that intersect it is a Lipschitz map. This shows that $X$ and $W$ are quasi-isometric, in particular $X$ is an HHS.
 \end{proof}

 	\par 

\subsection{Hyperbolicity of the factored contact graph}

 In this section, we prove that the factored contact graphs are uniformly hyperbolic spaces. The arguements and proofs which we use here are almost identical to \cite{BHS17},\cite{Val18} and the reader may skip this part.
\begin{proposition}[Lemma 2.10 of \cite{Val18}]
	Let $p,q,r\in V(X)$ be vertices of a quasi-median graph $X$ such that some hyperplane separates $q$ from $p$ and $r$. Then there exists a hyperplane $H$ separating $q$ from $p$ and $r$ and geodesics $\gamma_{p}, \gamma_{r}$ between $q,p$ and $q,r$ respectively such that $q$ is an endpoint of the edges of $\gamma_{p}$ and $\gamma_{r}$ dual to $H$. 
\end{proposition}
\begin{lemma}[Hierarchy path]
	Let $\widehat{CX}$ be the factored contact graph constructed in the proof of 4.10. Let $A,B\in V(\widehat{CX})$ and let $p\in N(A), q\in N(B)$. Then there exists a geodesic $A=A_{0},...,A_{m}=B$ in $ \widehat{CX}$ and vertices $p_{i}\in  N(A_{i-1})\cap N(A_{i})$ for $1\leq i \leq m$ such that $d_{X}(p,q)=\Sigma^{m}_{i=0} d_{X}(p_{i},p_{i+1})$.
\end{lemma}

\begin{proof}
	We first note that, by the defining property of adjacency of $\widehat{CX}$, if $A,A'\in \widehat{CX}$ are adjacent, then the associated gated subgraphs $N(A)$ and $N(A')$ intersect. (If $A=b_{F}$, a projection vertex we add, then we denote $N(A)$ as the graph $F$. By the construction, if $A=b_{F}$ and $A$ is adjacent to $A'$, then $A'$ is a hyperplane of $F$ and hence $N(A)$ intersects $N(A')$ in all cases.) Let $p_{0}=p, p_{m+1}=q$. Let $D=\Sigma^{m}_{i=0} d_{X}(p_{i},p_{i+1})$. Suppose $A_{i}$ and $p_{i}$ are chosen such that $D$ is minimized. We claim that $D=d_{X}(p,q)$. \par 
	Let $\gamma_{i}$ be a geodesic joining $p_{i}$ and $p_{i+1}$. Suppose for contradiction that $D>d_{X}(p,q)$, which means that the concatenation path $\gamma=\gamma_{0}\gamma_{1}...\gamma_{m}$ is not a geodesic. By proposition 4.4, there exists a hyperplane $H$ such that $\gamma$ intersects $H$ more than once. Since each subpath $\gamma_{i}$ are geodesic, there exists at least two distinct indices $0\leq i,j\leq m$ such that $\gamma_{i}$ intersects $H$. If $\vert i-j\vert >2$, then the subpath $A_{i}HA_{j}$ gives a short-cut, contradicting to the assumption that $A_{i}$ forms a geodesic in $\widehat{CX}$. Also if $\vert i-j\vert =2$, then the points lying in $A_{i}\cap N(H)$ ($A_{j}\cap N(H)$ respectively) which is closer to $p_{i}$ ($p_{j+1}$ respectively)  are joined by subpath of $\gamma$ that intersects $H$ more than twice, so replacing such subpath by the geodesic in $H$ and  $A_{i+1}$ by $H$ provides smaller $D$, contradicting to the minimality of $D$. Hence $\vert i-j\vert =1$. In this case, $H$ separates $p_{i+1}$ from $p_{i}$ and $p_{i+2}$ so using lemma 4.12, we may choose $H$ (also modify $\gamma_{i},\gamma_{i+1}$ if necessary) such that $H$ separates $p_{i+1}$ from $p_{i}$ and $p_{i+2}$ and the edges of $\gamma_{i},\gamma_{i+1}$ dual to $H$ have $p_{i+1}$ in common, that is there exists a triangle (clique) $C$ dual to $H$ and whose two sides $e_{1},e_{2}$ are edges of $\gamma_{i},\gamma_{i+1}$ respectively, intersecting at $p_{i+1}$. Since $\gamma_{i}\subset N(A_{i})$ and gated subgraphs contains its triangle, we have $C\subset N(A_{i})\cap N(N_{i+1})$. Again, we can replace the subpath $e_{1}e_{2}$ of $\gamma$ by a single edge contained in $C$ that joins end points of $e_{1}e_{2}$ which reduce $D$ by 1, contradicting to the minimality of $D$. Thus $D=d_{X}(p,q)$ as desired.
	
\end{proof}
We will prove the hyperbolicity of the factored contact graph by using the following criterion.
 \begin{proposition}[Bottleneck criterion,\cite{Man05}]
 	Let $X$ be a geomesic metric space. Suppose that there exists $\delta\geq 0$ such that for all $x,y\in X$, there exists a mid point $m=m(x,y)$ such that every path joining $x$ and $y$ intersects the $\delta$-ball about $m$. Then $X$ is $(26\delta, 16\delta)$-quasi-isometric to a tree.
 \end{proposition}
\begin{lemma}
	The factored contact graphs are uniform quasi-tree.
\end{lemma}
\begin{proof}
	Let $A,B\in\widehat{CX}$. If $d(A,B)<10$, then the 5-bottleneck criterion is automatically satisfied, so assume that $d(A,B)\geq 10$. In this case, the associated gated subgraphs for $A$ and $B$ are disjoint and we can find the closest points $p\in N(A),q\in N(B)$ such that any hyperplane separating $p$ and $q$ separates $N(A)$ and $N(B)$. Using lemma 4.13, find the hierarchy path $A=A_{0},A_{1},...,A_{m}=B$ and $p=p_{0},p_{1},...,p_{m+1}=q$. Let $M=A_{i}$ be a point of $V(\widehat{CX})$ which is closest to the midpoint of the hierarchy path. Then there exists hyperplanes $H$ separating $p_{i}$ and $p_{i+1}$. Since the hierarchy path is a geodesic in $X$, the hyperplane $H$ separates $p$, $q$ and hence separates $N(A)$ and $N(B)$ by our construction. Also $H$ is adjacent to $M$. Now suppose $\sigma\subset \widehat{CX}$ is any edge path joining $A$ and $B$. Since $H$ separates $N(A)$ and $N(B)$ and the adjacency in $\widehat{CX}$ implies non-empty intersection between their associated gated subgraphs, we conclude that there exists some vertex, say $C$, in the path $\sigma$ such that $H$ separates some points in $C$. It shows that $H$ is adjacent to $C$ and $M$, hence $C$ has distance at most $3$ from $M$, which shows the 3-bottlenect criterion of $V(\widehat{CX})$. Using the similar arguement, we can easily extend it to the 5-bottleneck criterion of  $\widehat{CX}$.
\end{proof}

\section*{Acknowledgement}
 The authors thank the referee for their careful reading of the manuscript and for giving various useful comments.
 \medskip
 
 \bibliographystyle{abbrv}
 \bibliography{bibilography}
\end{document}